\documentclass[preprint,12pt]{elsarticle}
\usepackage{amsmath, amssymb, amsthm, algorithm, graphicx, url, hyperref, mathtools, tensor}
\usepackage{xcolor}
\usepackage{tcolorbox}
\tcbuselibrary{breakable}
\usepackage{lipsum}  
\usepackage{mathrsfs}
\newtcolorbox{mycomment}{
    colback=yellow!10,
    colframe=yellow!50!black,
    title=Comment,
    fonttitle=\bfseries,
    boxrule=1pt
}
\usepackage[capitalize]{cleveref}
\usepackage{geometry}
\usepackage{algpseudocode}
\usepackage{caption}
\usepackage{subcaption}
\usepackage{makecell}
\usepackage{tikz}
\usepackage{pgfplots}
\pgfplotsset{compat=1.18}
\usepackage{booktabs}
\DeclareMathOperator*{\argmin}{arg\,min}
\newtheorem{theorem}{Theorem}[section]
\newtheorem{lemma}[theorem]{Lemma}
\newtheorem{corollary}[theorem]{Corollary}

\newtheorem{definition}[theorem]{Definition}
\newtheorem{remark}[theorem]{Remark}
\newtheorem{assumption}[theorem]{Assumption}

\newcommand{\R}{\mathbb{R}}

\newcommand{\Input}{\item[\textbf{Input:}]}
\newcommand{\Output}{\item[\textbf{Output:}]}

\begin{document}

\begin{frontmatter}

\title{\texorpdfstring{Adaptive Randomized Pivoting for Tensor Cross Approximation
in the T-Product Framework}{Adaptive Randomized Pivoting for Tensor Cross Approximation in the T-Product Framework}}



\author[inst1]{Ahmadsho Akdodshoev}


\author[inst2]{Valentin Leplat%
\texorpdfstring{\corref{cor1}\fnref{fn1}}{}}
\ead{valentin.leplat@gmail.com}

\author[inst1]{Salman Ahmadi-Asl%
\texorpdfstring{\fnref{fn1}}{}}
\ead{s.ahmadiasl@innopolis.ru}

\address[inst1]{Innopolis University, 420500 Innopolis, Russia}
\address[inst2]{Department of Mathematics and Operational Research, Universit\'e de Mons, Rue de Houdain 9, 7000 Mons, Belgium}

\cortext[cor1]{Corresponding author.}

\fntext[fn1]{Valentin Leplat and Salman Ahmadi-Asl contributed equally to this work.}

\begin{abstract}
This paper studies extensions of adaptive randomized pivoting (ARP), recently introduced for matrix column subset selection, to tensors in the t-product framework. We propose two constructions. ARP-T-CUR applies matrix ARP-cross at the nonredundant Fourier frequencies and uses conjugate symmetry to preserve real-valued tensors. Assuming that the selected intersections are nonsingular almost surely, the matrix theory gives a direct expected-error bound for this Fourier-slicewise approximation. T-ARP instead selects common lateral and horizontal slices, with the same indices used at every Fourier frequency. This common-index constraint requires a new analysis. Under frequency-alignment and frequency-wise rank-growth assumptions, we prove an expected-error bound and recover the matrix ARP factor \(r+1\) when the frequency-wise sampling distributions are aligned. We also derive bounds for tensor cross approximation and t-DEIM. The numerical experiments show that both proposed methods give accurate tensor approximations and perform favorably against the considered slice-selection baselines. A separate full-enumeration experiment verifies the assumptions and the expected-error bound for common-index T-ARP.

\end{abstract}


\end{frontmatter}

\section{Introduction}
\label{sec:introduction}

Tensor decompositions are widely used to represent and process multidimensional data. Important examples include the Higher-Order Singular Value Decomposition (HOSVD)~\cite{de2000multilinear}, the tensor singular value decomposition (T-SVD) in the t-product framework~\cite{kilmer2011factorization}, tensor train (TT) decompositions~\cite{oseledets2011tensor}, tensor ring (TR) representations~\cite{zhao2016tensor}, and more recently the tubal tensor train (TTT) decomposition~\cite{ahmadiasl2026newtensornetworktubal}. These methods have been applied to image completion~\cite{song2019tensor,asante2023image}, biomedical signals~\cite{duan2025paramps}, quantum simulations~\cite{bortone2025simple,brice2025numerical}, and collaborative filtering~\cite{de2025weighted,zhou2023tensor}. Their computation often relies on repeated matrix factorizations, in particular singular value decompositions.

For large tensors, deterministic decompositions can be expensive in time and memory. This motivates randomized methods based on sampling and sketching.

Within matrix approximation, column subset selection aims to represent a matrix using a small subset of its actual columns. The Adaptive Randomized Pivoting (ARP) method of Cortinovis and Kressner~\cite{cortinovis2024adaptive} sequentially samples columns from an evolving row-space basis and gives an \((r+1)\) bound for the expected squared projection error. Motivated by this result, we study extensions of ARP to the t-product framework. Related randomized and adaptive methods for t-product tensor approximation were proposed in~\cite{tarzanagh2018fast,ahmadi2024adaptive,ahmadi2024robust,ahmadi2024randomized}.

Our main contributions are as follows:
\begin{itemize}
    \item We present two extensions of ARP to the t-product framework: a Fourier-slicewise construction, called ARP-T-CUR, and a common-index tensor construction, called T-ARP.
    \item For ARP-T-CUR, we obtain a direct slicewise expected-error bound by applying the matrix ARP-cross result frequency by frequency. For T-ARP, we identify the additional coupling created by using the same indices at all Fourier frequencies and derive an expected-error bound under explicit frequency-alignment and frequency-wise rank-growth assumptions. The matrix ARP factor \(r+1\) is recovered in the perfectly aligned case.
    \item We derive corresponding bounds for two-stage tensor cross approximation and t-DEIM.
    \item We provide Python/JAX implementations and numerical tests with two complementary purposes: practical projection-based benchmarks show that both proposed methods perform favorably against the considered slice-selection baselines, while a small full-enumeration experiment directly validates the common-index theorem and shows that its aligned factor can be sharp.
\end{itemize}

The theory is developed for third-order tensors. The video experiment uses the higher-order product of Martin et al.~\cite{Martin2013}; we treat this as a separate implementation extension rather than claiming that every third-order proof transfers automatically.

The paper is organized as follows. Section~\ref{sec:preliminaries} reviews the t-product, T-SVD, and T-QR decomposition. Section~\ref{SEc:ARP} recalls matrix ARP. Section~\ref{sec:tsvd} introduces the two tensor extensions and gives their error analyses. Section~\ref{SEc:NE} reports the numerical experiments. Section~\ref{sec:conclusion} concludes the paper.

\section{Preliminaries}
\label{sec:preliminaries}
In this section, we establish the notation and mathematical foundations used for the T-SVD. Calligraphic letters denote tensors, uppercase letters denote matrices, and lowercase letters denote vectors. The Frobenius norm of a tensor or matrix is denoted by \(\|\cdot\|_F\). The T-SVD is built on the t-product framework introduced by Kilmer and Martin~\cite{kilmer2011factorization}.

Let \(\mathcal{X} \in \mathbb{R}^{n_1 \times n_2 \times n_3}\) denote a third-order tensor. We adopt the following standard notation:
\begin{itemize}
    \item \textbf{Fibers:} A fiber is a one-dimensional slice obtained by fixing all indices but one. For a third-order tensor, we have:
    \begin{itemize}
        \item \textit{Column fiber} (mode-1): \(\mathcal{X}(:, j, k)\).
        \item \textit{Row fiber} (mode-2): \(\mathcal{X}(i, :, k)\).
        \item \textit{Tube fiber} (mode-3): \(\mathcal{X}(i, j, :)\).
    \end{itemize}
    \item \textbf{Slices:} A slice is a two-dimensional section of a tensor:
    \begin{itemize}
        \item \textit{Frontal slice}: \(X^{(k)} = \mathcal{X}(:, :, k)\) for \(k = 1,2,\ldots,n_3\); it is obtained by fixing the third index.
        \item \textit{Lateral slice}: \(\mathcal{X}(:, j, :)\) for \(j = 1,2,\ldots,n_2\).
        \item \textit{Horizontal slice}: \(\mathcal{X}(i, :, :)\) for \(i = 1,2,\ldots,n_1\); it is obtained by fixing the first index.
    \end{itemize}
\end{itemize}

The t-product is a multiplication operation between two third-order tensors of appropriate dimensions. To define it, we first need the concepts of block circulant matrices and the unfold operator.

\begin{definition}[Unfold and Fold Operators]
For a tensor \(\mathcal{X} \in \mathbb{R}^{n_1 \times n_2 \times n_3}\), the \emph{unfold} operator maps \(\mathcal{X}\) to a block matrix of size \(n_1 n_3 \times n_2\):
\begin{equation}
\texttt{unfold}(\mathcal{X}) = 
\begin{bmatrix}
{X}^{(1)} \\
{X}^{(2)} \\
\vdots \\
{X}^{(n_3)}
\end{bmatrix},
\end{equation}
where \({X}^{(k)}\) is the \(k\)-th frontal slice. The inverse \emph{fold} operator, \(\texttt{fold}(\texttt{unfold}(\mathcal{X})) = \mathcal{X}\), reshapes the block matrix back into a tensor.
\end{definition}

\begin{definition}[Block Circulant Matrix]
The \emph{block circulant matrix} of a tensor \(\mathcal{X} \in \mathbb{R}^{n_1 \times n_2 \times n_3}\) is an \(n_1 n_3 \times n_2 n_3\) block matrix defined as:
\begin{equation}
\texttt{bcirc}(\mathcal{X}) = 
\begin{bmatrix}
{X}^{(1)} & {X}^{(n_3)} & {X}^{(n_3-1)} & \dots & {X}^{(2)} \\
{X}^{(2)} & {X}^{(1)} & {X}^{(n_3)} & \dots & {X}^{(3)} \\
\vdots & \vdots & \vdots & \ddots & \vdots \\
{X}^{(n_3)} & {X}^{(n_3-1)} & {X}^{(n_3-2)} & \dots & {X}^{(1)}
\end{bmatrix}.
\end{equation}
\end{definition}

\begin{definition}[T-Product]
Let \(\mathcal{X} \in \mathbb{R}^{n_1 \times n_2 \times n_3}\) and \(\mathcal{Y} \in \mathbb{R}^{n_2 \times \ell \times n_3}\). The \emph{t-product} \(\mathcal{Z} = \mathcal{X} * \mathcal{Y} \in \mathbb{R}^{n_1 \times \ell \times n_3}\) is defined as:
\begin{equation}
\mathcal{Z} = \texttt{fold}(\texttt{bcirc}(\mathcal{X}) \cdot \texttt{unfold}(\mathcal{Y})).
\end{equation}
\end{definition}

The t-product can be understood as matrix multiplication in which the scalar multiplication operation is replaced by circular convolution between tube fibers. More precisely, the t-product between tensors $\mathcal{X} \in \mathbb{R}^{m \times n \times p}$ and $\mathcal{Y} \in \R^{n \times q \times p}$ yields $\mathcal{Z} \in \R^{m \times q \times p}$ defined via:
\begin{equation}
\mathcal{Z}(i,j,:) = \sum_{k=1}^n \mathcal{X}(i,k,:) \ast \mathcal{Y}(k,j,:),
\end{equation}
where $\ast$ denotes circular convolution between tubes (vectors of length $p$). Equivalently, by the convolution theorem, the t-product corresponds to face-wise matrix multiplication in the Fourier domain.

A fundamental fact behind the t-product is its reduction to slicewise matrix multiplication in the Fourier domain. Let \(\widehat{\mathcal X}=\texttt{fft}(\mathcal X,[],3)\), where the FFT is applied along the third mode. Thus, for each \((i,j)\),
\begin{equation}
\widehat{\mathcal X}(i,j,:)=\texttt{fft}(\mathcal X(i,j,:)).
\end{equation}

\begin{theorem}[Block Diagonalization]
Let \(F_{n_3}\) denote the matrix associated with the unnormalized DFT used by \texttt{fft}, so that \(F_{n_3}^{-1}=n_3^{-1}F_{n_3}^H\). Then
\begin{equation}
(F_{n_3}\otimes I_{n_1})\,\texttt{bcirc}(\mathcal X)\,(F_{n_3}^{-1}\otimes I_{n_2})
=
\operatorname{bdiag}\!\left(\widehat X^{(1)},\ldots,\widehat X^{(n_3)}\right).
\end{equation}
Here \(\widehat X^{(k)}\) denotes the \(k\)-th frontal slice of \(\widehat{\mathcal X}\).
\end{theorem}

Consequently, if \(\mathcal Z=\mathcal X*\mathcal Y\), then
\begin{equation}
\widehat Z^{(k)}=\widehat X^{(k)}\widehat Y^{(k)},\qquad k=1,\ldots,n_3.
\end{equation}
For the same unnormalized FFT convention, Parseval's identity gives
\begin{equation}
\|\mathcal X\|_F^2=\frac{1}{n_3}\sum_{\ell=1}^{n_3}\|\widehat X^{(\ell)}\|_F^2.
\end{equation}

\begin{algorithm}[H]
\caption{T-Product via Fourier Domain (Half-Slice Version)}
\label{alg:tproduct_half}
\begin{algorithmic}[1]
\Input Tensors \(\mathcal X\in\mathbb R^{n_1\times n_2\times n_3}\) and \(\mathcal Y\in\mathbb R^{n_2\times q\times n_3}\)
\Output Tensor \(\mathcal Z=\mathcal X*\mathcal Y\in\mathbb R^{n_1\times q\times n_3}\)
\State Compute \(\widehat{\mathcal X}=\texttt{fft}(\mathcal X,[],3)\).
\State Compute \(\widehat{\mathcal Y}=\texttt{fft}(\mathcal Y,[],3)\).
\State Set \(L=\lfloor n_3/2\rfloor+1\).
\For{\(k=1,2,\ldots,L\)}
    \State \(\widehat Z^{(k)} = \widehat X^{(k)} \widehat Y^{(k)}\).
\EndFor
\For{\(k=L+1,\ldots,n_3\)}
    \State \(\widehat Z^{(k)} = \overline{\widehat Z^{(n_3-k+2)}}\).
\EndFor
\State Assemble \(\widehat{\mathcal Z}\) from \(\{\widehat Z^{(k)}\}_{k=1}^{n_3}\).
\State \(\mathcal Z = \texttt{ifft}(\widehat{\mathcal Z},[],3)\).
\State \Return \(\mathcal Z\).
\end{algorithmic}
\end{algorithm}


\begin{definition}[Tensor Transpose]
The transpose of a tensor \(\mathcal{X} \in \mathbb{R}^{n_1 \times n_2 \times n_3}\), denoted \(\mathcal{X}^{\top}\), is the \(n_2 \times n_1 \times n_3\) tensor obtained by transposing each frontal slice and then reversing the order of slices 2 through \(n_3\). More formally, for \(k = 1\):
\begin{equation}
({X}^{\top})^{(1)} = ({X}^{(1)})^{\top},
\end{equation}
and for \(k = 2,3,\ldots,n_3\):
\begin{equation}
({X}^{\top})^{(k)} = ({X}^{(n_3 - k + 2)})^{\top}.
\end{equation}
\end{definition}

For a complex tensor \(\mathcal X\), its conjugate transpose \(\mathcal X^*\) is defined in the same way, replacing each matrix transpose by a conjugate transpose. Equivalently, its Fourier slices satisfy
\[
\widehat{(\mathcal X^*)}^{(\ell)}
=
\bigl(\widehat X^{(\ell)}\bigr)^H,
\qquad \ell=1,2,\ldots,n_3.
\]
For real tensors, \(\mathcal X^*=\mathcal X^\top\).

\begin{definition}[Identity Tensor]\label{def:identity-tensor}
The \emph{identity tensor} \(\mathcal{I}_n \in \mathbb{R}^{n \times n \times n_3}\) is a tensor whose first frontal slice is the \(n \times n\) identity matrix, and all other frontal slices are zero matrices. For any tensor \(\mathcal{X}\) of compatible dimensions, \(\mathcal{X} * \mathcal{I}_n = \mathcal{X}\) and \(\mathcal{I}_n * \mathcal{X} = \mathcal{X}\).
\end{definition}

\begin{definition}[Orthogonal Tensor]
A tensor \(\mathcal{Q} \in \mathbb{R}^{n \times n \times n_3}\) is \emph{orthogonal} if it satisfies:
\begin{equation}
\mathcal{Q}^{\top} * \mathcal{Q} = \mathcal{Q} * \mathcal{Q}^{\top} = \mathcal{I}.
\end{equation}
In the Fourier domain, this is equivalent to each frontal slice \(\widehat Q^{(k)}\) being unitary, that is, \((\widehat Q^{(k)})^H\widehat Q^{(k)}=I\).
\end{definition}

\begin{definition}[f-diagonal Tensor]
A tensor is called \emph{f-diagonal} if each of its frontal slices is a diagonal matrix. For a third-order tensor \(\mathcal{S}\in \mathbb{R}^{n_1 \times n_2 \times n_3}\), this means \({S}^{(k)}\) is diagonal for all \(k = 1,2,\ldots, n_3\).
\end{definition}

\begin{definition}[Upper Triangular Tensor]
A tensor $\mathcal{R} \in \R^{n_1 \times n_2 \times n_3}$ is \emph{upper triangular} if each of its frontal slices is an upper triangular matrix.
\end{definition}

\begin{definition}[Inverse of a Tensor]
Let \(\mathcal{A} \in \mathbb{R}^{n \times n \times n_3}\) be a tensor. The \emph{inverse} of \(\mathcal{A}\) under the t-product, denoted \(\mathcal{A}^{-1} \in \mathbb{R}^{n \times n \times n_3}\), is the unique tensor satisfying
\[
\mathcal{A} * \mathcal{A}^{-1} = \mathcal{I}_n \quad \text{and} \quad \mathcal{A}^{-1} * \mathcal{A} = \mathcal{I}_n,
\]
where \(\mathcal I_n\in\mathbb R^{n\times n\times n_3}\) is the identity tensor of Definition~\ref{def:identity-tensor}. The tensor \(\mathcal{A}\) is said to be \emph{invertible} if such an inverse exists. See Algorithm~\ref{alg:tinverse_half} for its calculation.
\end{definition}

\begin{algorithm}[H]
\caption{Tensor Inverse via Fourier Domain (Half-Slice Version)}
\label{alg:tinverse_half}
\begin{algorithmic}[1]
\Input Tensor $\mathcal{A} \in \mathbb{R}^{n \times n \times n_3}$ (square in first two dimensions);
\Output Tensor $\mathcal{A}^{-1} \in \mathbb{R}^{n \times n \times n_3}$ such that $\mathcal{A} * \mathcal{A}^{-1} = \mathcal{I}_n$;
\State Compute $\widehat{\mathcal{A}} = \texttt{fft}(\mathcal{A}, [], 3)$;
\State Set \(L=\lfloor n_3/2\rfloor+1\).
\For{$i = 1, 2, \ldots, L$}
    \If{$\widehat{A}^{(i)}$ is singular}
        \State \textbf{error}: Tensor is not invertible;
    \Else
        \State $\widehat{A^{-1}}^{(i)} = (\widehat{A}^{(i)})^{-1}$ \Comment{Matrix inverse};
    \EndIf
\EndFor
\For{$i = L+1, \ldots, n_3$}
    \State $\widehat{A^{-1}}^{(i)} = \overline{\widehat{A^{-1}}^{(n_3-i+2)}}$ \Comment{Conjugate symmetry};
\EndFor
\State $\widehat{\mathcal{A}^{-1}} \gets$ assemble from $\{\widehat{A^{-1}}^{(i)}\}_{i=1}^{n_3}$;
\State $\mathcal{A}^{-1} = \texttt{ifft}(\widehat{\mathcal{A}^{-1}}, [], 3)$;
\State \Return $\mathcal{A}^{-1}$;
\end{algorithmic}
\end{algorithm}


\begin{definition}[Pseudoinverse of a Tensor]
Let \(\mathcal{A} \in \mathbb{R}^{n_1 \times n_2 \times n_3}\) be a tensor. The \emph{pseudoinverse} of \(\mathcal{A}\) under the t-product, denoted \(\mathcal{A}^{\dagger} \in \mathbb{R}^{n_2 \times n_1 \times n_3}\), is the unique tensor satisfying the following four Moore-Penrose conditions:
\[
\begin{aligned}
\mathcal{A} * \mathcal{A}^{\dagger} * \mathcal{A} &= \mathcal{A}, \\
\mathcal{A}^{\dagger} * \mathcal{A} * \mathcal{A}^{\dagger} &= \mathcal{A}^{\dagger}, \\
(\mathcal{A} * \mathcal{A}^{\dagger})^* &= \mathcal{A} * \mathcal{A}^{\dagger}, \\
(\mathcal{A}^{\dagger} * \mathcal{A})^* &= \mathcal{A}^{\dagger} * \mathcal{A},
\end{aligned}
\]
where \((\cdot)^*\) denotes the tensor conjugate transpose. If every Fourier slice \(\widehat A^{(\ell)}\) has full column rank, then \(\mathcal A^\dagger=(\mathcal A^\top*\mathcal A)^{-1}*\mathcal A^\top\). If every Fourier slice has full row rank, then \(\mathcal A^\dagger=\mathcal A^\top*(\mathcal A*\mathcal A^\top)^{-1}\). In general, the pseudoinverse is computed slicewise by matrix SVD, as in Algorithm~\ref{alg:tpseudoinverse_half}.
\end{definition}

\begin{algorithm}[H]
\caption{Tensor Pseudoinverse via Fourier Domain (Half-Slice Version)}
\label{alg:tpseudoinverse_half}
\begin{algorithmic}[1]
\Input Tensor $\mathcal{A} \in \mathbb{R}^{n_1 \times n_2 \times n_3}$;
\Output Tensor $\mathcal{A}^{\dagger} \in \mathbb{R}^{n_2 \times n_1 \times n_3}$ (Moore-Penrose pseudoinverse);
\State Compute $\widehat{\mathcal{A}} = \texttt{fft}(\mathcal{A}, [], 3)$;
\State Set \(L=\lfloor n_3/2\rfloor+1\).
\For{$i = 1, 2, \ldots, L$}
    \State Compute $\widehat{A^{\dagger}}^{(i)} = (\widehat A^{(i)})^\dagger$ by a matrix SVD.
\EndFor
\For{$i = L+1, \ldots, n_3$}
    \State $\widehat{A^{\dagger}}^{(i)} = \overline{\widehat{A^{\dagger}}^{(n_3-i+2)}}$ \Comment{Conjugate symmetry};
\EndFor
\State $\widehat{\mathcal{A}^{\dagger}} \gets$ assemble from $\{\widehat{A^{\dagger}}^{(i)}\}_{i=1}^{n_3}$;
\State $\mathcal{A}^{\dagger} = \texttt{ifft}(\widehat{\mathcal{A}^{\dagger}}, [], 3)$;
\State \Return $\mathcal{A}^{\dagger}$;
\end{algorithmic}
\end{algorithm}


With the t-product framework established, we can now define the T-SVD.

\begin{definition}[T-SVD]
Let \(\mathcal X\in\mathbb R^{n_1\times n_2\times n_3}\). A T-SVD of \(\mathcal X\) is
\begin{equation}
\mathcal X=\mathcal U*\mathcal S*\mathcal V^\top,
\end{equation}
where \(\mathcal U\in\mathbb R^{n_1\times n_1\times n_3}\) and \(\mathcal V\in\mathbb R^{n_2\times n_2\times n_3}\) are orthogonal tensors, and \(\mathcal S\in\mathbb R^{n_1\times n_2\times n_3}\) is f-diagonal.
\end{definition}

\begin{remark}[Computation via Fourier Domain]
Compute \(\widehat{\mathcal X}=\texttt{fft}(\mathcal X,[],3)\). For each frequency \(\ell\), take the matrix SVD
\begin{equation}
\widehat X^{(\ell)}=\widehat U^{(\ell)}\widehat S^{(\ell)}(\widehat V^{(\ell)})^H.
\end{equation}
For a real tensor, conjugate Fourier pairs may be coupled so that the inverse FFT of the factors is real. The tensors \(\mathcal U,\mathcal S,\mathcal V\) are then obtained by applying the inverse FFT along the third mode.
\end{remark}

\begin{definition}[Tubal Rank]
The tubal rank of \(\mathcal X\) is the number of nonzero singular tubes of \(\mathcal S\). Equivalently,
\begin{equation}
\operatorname{rank}_t(\mathcal X)=\max_{1\leq \ell\leq n_3}\operatorname{rank}(\widehat X^{(\ell)}).
\end{equation}
Thus, the tubal rank counts nonzero diagonal tubes, not individual scalar entries of the spatial-domain tensor \(\mathcal S\).
\end{definition}

The best tubal-rank-\(r\) approximation in the Frobenius norm is obtained by truncating the T-SVD:
\begin{equation}
\mathcal X_r=\mathcal U(:,1:r,:)*\mathcal S(1:r,1:r,:)*\mathcal V(:,1:r,:)^\top.
\end{equation}
It solves
\begin{equation}
\mathcal X_r\in\argmin_{\operatorname{rank}_t(\mathcal B)\leq r}\|\mathcal X-\mathcal B\|_F,
\end{equation}
and its squared error is
\begin{equation}
\|\mathcal X-\mathcal X_r\|_F^2=\sum_{k=r+1}^{\min(n_1,n_2)}\|\mathcal S(k,k,:)\|_2^2.
\end{equation}
Figure~\ref{fig:TSVD} illustrates the T-SVD and its truncated form.

\begin{figure}[t]
    \centering
\includegraphics[width=0.70\textwidth]{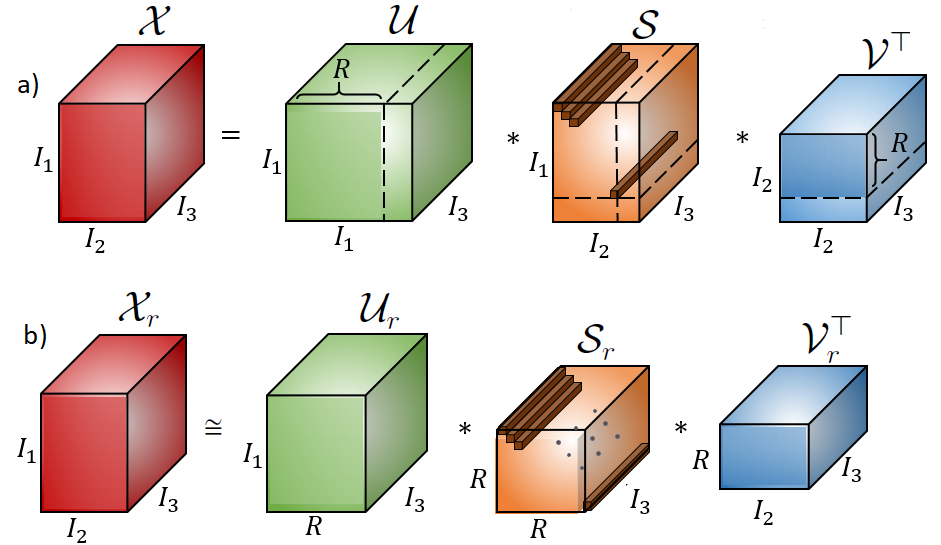}
    \caption{Tensor SVD and its truncated version.}
    \label{fig:TSVD}
\end{figure}

The T-QR decomposition is a factorization that expresses a tensor as the t-product of an orthogonal tensor and an upper triangular tensor.

\begin{definition}[T-QR Decomposition]
Let $\mathcal{X} \in \R^{n_1 \times n_2 \times n_3}$ with $n_1 \ge n_2$. The \emph{T-QR decomposition} of $\mathcal{X}$ is given by
\[
\mathcal{X} = \mathcal{Q} * \mathcal{R},
\]
where:
\begin{itemize}
    \item $\mathcal{Q} \in \R^{n_1 \times n_2 \times n_3}$ is a tensor with orthonormal lateral slices: $\mathcal{Q}^{\top} * \mathcal{Q} = \mathcal{I}$ (the identity tensor of size $n_2 \times n_2 \times n_3$);
    \item $\mathcal{R} \in \R^{n_2 \times n_2 \times n_3}$ is an upper triangular tensor (each frontal slice is upper triangular).
\end{itemize}
If $n_1 > n_2$, the decomposition is often called the \emph{thin} or \emph{economy} T-QR.
\end{definition}

Algorithm~\ref{alg:tqr} summarizes this process.

\begin{algorithm}[H]
\caption{T-QR Decomposition via Householder Transformations}
\label{alg:tqr}
\begin{algorithmic}[1]
\Input Tensor \(\mathcal{X} \in \R^{n_1 \times n_2 \times n_3}\) with \(n_1 \ge n_2\)
\Output Orthogonal tensor \(\mathcal{Q} \in \R^{n_1 \times n_2 \times n_3}\) and upper triangular tensor \(\mathcal{R} \in \R^{n_2 \times n_2 \times n_3}\)
\State Compute \(\widehat{\mathcal{X}} = \texttt{fft}(\mathcal{X}, [], 3)\) and set \(L=\lfloor n_3/2\rfloor+1\).
\For{\(\ell = 1, 2, \ldots, L\)}
    \State \(\widehat A\gets\widehat X^{(\ell)}\).
    \If{\(\ell=1\) or \(\ell=n_3-\ell+2\)}
        \State \([\widehat Q^{(\ell)},\widehat R^{(\ell)}]\gets
        \texttt{householder\_qr}(\operatorname{real}(\widehat A))\).
    \Else
        \State \([\widehat Q^{(\ell)},\widehat R^{(\ell)}]\gets
        \texttt{householder\_qr}(\widehat A)\).
        \State Set \(\ell'=n_3-\ell+2\), \(\widehat Q^{(\ell')}=\overline{\widehat Q^{(\ell)}}\), and \(\widehat R^{(\ell')}=\overline{\widehat R^{(\ell)}}\).
    \EndIf
\EndFor
\State Assemble \(\widehat{\mathcal Q}\) and \(\widehat{\mathcal R}\) from their frontal slices.
\State \(\mathcal Q=\operatorname{real}(\texttt{ifft}(\widehat{\mathcal Q},[],3))\).
\State \(\mathcal R=\operatorname{real}(\texttt{ifft}(\widehat{\mathcal R},[],3))\).
\State \Return \(\mathcal Q,\mathcal R\)
\end{algorithmic}
\end{algorithm}

\section{Adaptive Randomized Pivoting (ARP)}\label{SEc:ARP}
Adaptive Randomized Pivoting (ARP) is a randomized algorithm for the Column Subset Selection Problem (CSSP) proposed by Cortinovis and Kressner \cite{cortinovis2024adaptive}. Given a matrix \(X \in \mathbb{R}^{m \times n}\) and a target rank \(r \ll \min\{m,n\}\), CSSP aims to select \(r\) column indices \(J = (j_1,j_2,\ldots,j_r)\) such that the selected columns \(X(:,J)\) approximately span the column space of \(X\). The approximation error is measured as
\[
\|X - \Pi_J X\|_F,
\]
where
\[
\Pi_J
=
C(C^T C)^\dagger C^T,
\qquad
C=X(:,J),
\]
is the orthogonal projector onto \(\operatorname{span}(X(:,J))\)
and \(\|\cdot\|_F\) denotes the Frobenius norm.

ARP relies on two key ideas:
\begin{enumerate}
    \item \textbf{Subspace sampling}: Sampling probabilities are derived from an orthonormal basis \(V \in \mathbb{R}^{n \times r}\) that approximates the row space of \(X\);
    \item \textbf{Adaptive updating}: After each selection, the basis is orthogonalized against the selected row, making the sampling adaptive.
\end{enumerate}

The index selection uses only the basis \(V\), not the matrix \(X\). This is useful in the Discrete Empirical Interpolation Method (DEIM)~\cite{chaturantabut2010nonlinear}, where the function to be approximated is not available during index selection.

The ARP algorithm processes the orthonormal basis \(V \in \mathbb{R}^{n \times r}\) iteratively. At each step \(k = 1,2,\ldots,r\), it samples a row index \(j_k\) with probability proportional to the squared norm of the current row, then removes that row's contribution via an orthogonal projection. Algorithm~\ref{alg:arp} presents the stable Householder-based implementation.

\begin{algorithm}[H]
\caption{Adaptive Randomized Pivoting for CSSP (ARP)}
\label{alg:arp}
\begin{algorithmic}[1]
\Input  Orthonormal basis \(V \in \mathbb{R}^{n_2 \times r}\) (row space approximation);
\Output Index set \(J = (j_1,j_2,\ldots,j_r)\)
\State Initialize \(J = ()\) and \(W_0 = V\);
\For{\(k = 1,2,\ldots,r\)}
    \State Compute probabilities: \(p_j = \frac{\|W_{k-1}(j, k:r)\|_2^2}{r - k + 1}\) for \(j = 1,2,\ldots,n_2\);
    \State Sample \(j_k \sim \operatorname{Categorical}(p_1,p_2,\ldots,p_{n_2})\);
    \State Append \(j_k\) to \(J\);
    \State Update \(W_k \leftarrow W_{k-1} Q_k\), where \(Q_k\) is a Householder reflector that annihilates \(W_{k-1}(j_k, k+1:r)\);
\EndFor
\State \Return \(J\)
\end{algorithmic}
\end{algorithm}

The algorithm maintains a transformed basis \(W_k\) that remains orthonormal throughout. The probabilities in Line~3 satisfy \(\sum_{j=1}^{n_2} p_j = 1\) because \(\|W_{k-1}(:,k:r)\|_F^2 = r - k + 1\). The Householder update ensures numerical stability and prevents the same index from being selected twice.

\begin{algorithm}[H]
\caption{Adaptive Randomized Pivoting for cross approximation (ARP-cross)}
\label{alg:arpcross}
\begin{algorithmic}[1]
\Input Matrix \(X \in \mathbb{R}^{n_1 \times n_2}\) and an orthonormal basis \(V \in \mathbb{R}^{n_2 \times r}\) defining a row space approximation;
\Output Index sets \(I = (i_1,i_2,\ldots,i_r)\) and \(J = (j_1,j_2,\ldots,j_r)\) defining a cross approximation \(X \approx X(:,J) X(I,J)^{-1} X(I,:)\);
\State Obtain index set \(J\) by applying Algorithm \ref{alg:arp} to \(V\);
\State Compute an orthonormal basis \(Q_J\) of \(X(:,J)\) by a QR decomposition;
\State Obtain index set \(I\) by applying Algorithm \ref{alg:arp} to \(Q_J\);
\State \Return \(I,J\)
\end{algorithmic}
\end{algorithm}

The main result of Cortinovis and Kressner is an \((r+1)\) bound for the expected squared projection error. Define the oblique projector
\[
\widetilde{\Pi}_J = I - E_J (V^T E_J)^{-1} V^T,
\]
where \(E_J = [e_{j_1},e_{j_2}, \ldots, e_{j_r}]\). The next theorem summarizes the main results shown in \cite{cortinovis2024adaptive}.

\begin{theorem}[Cortinovis \& Kressner, 2026]\label{th_cor_kres}
Let \(X \in \mathbb{R}^{n_1 \times n_2}\) and let \(V \in \mathbb{R}^{n_2 \times r}\) be an orthonormal basis. The random index set \(J\) returned by Algorithm~\ref{alg:arp} satisfies

{\color{black}
\[
\mathbb{E}\left[\|X-\Pi_JX\|_F^2\right]
\leq
\mathbb{E}\left[
\|X-X(:,J)V(J,:)^{-T}V^T\|_F^2
\right]
=
(r+1)\|X-XVV^T\|_F^2 .
\]
}
\end{theorem}

Several important consequences follow:
\begin{itemize}
    \item By Jensen's inequality, \(\mathbb{E}(\|X - \Pi_J X\|_F) \le \sqrt{r+1} \|X - X V V^T\|_F\).
    \item When \(V = V_{\mathrm{opt}}\) (the right singular vectors of \(X\)), the bound becomes \((r+1)(\sigma_{r+1}^2 + \cdots + \sigma_n^2)\), matching the optimal existence result of Deshpande et al.~\cite{deshpande2006matrix}.
    \item Markov's inequality yields tail bounds: with probability \(\ge 0.99\), the error is at most \(10\sqrt{r+1}\|X - X V V^T\|_F\).
\end{itemize}

The ARP framework extends naturally to several related problems:
\begin{itemize}
    \item DEIM (Discrete Empirical Interpolation Method): For a function \(f \approx V V^T f\), ARP selects indices \(I\) such that \(f \approx V (E_I^T V)^{-1} f(I)\). Corollary~3.1 of the paper shows
\[
\mathbb{E}[\|(E_I^T V)^{-1}\|_F^2] = r(n_2 - r + 1), \quad
\mathbb{E}[\|(E_I^T V)^{-1}\|_2^2] \le 1 + r(n_2 - r).
\]

\item Cross (Skeleton) Approximation: For general matrices, selecting both rows and columns yields the approximation \(X \approx X(:,J) X(I,J)^{-1} X(I,:)\). Using ARP twice (once with \(V\) for columns, once with an orthonormal basis of \(X(:,J)\) for rows) gives
\[
\mathbb{E}[\|X - X(:,J)X(I,J)^{-1}X(I,:)\|_F^2] \le (r+1)^2 \|X - X V V^T\|_F^2.
\]

\item Nyström Approximation for SPSD Matrices: For symmetric positive semidefinite \(X\), choosing \(I=J\) yields the Nyström approximation. The Gram correspondence \(\|X - X(:,J)X(J,J)^{-1}X(J,:)\|_* = \|(I - \Pi_J) B\|_F^2\) (where \(B^T B = X\) and \(\|\cdot\|_*\) is the nuclear norm) leads to
\[
\mathbb{E}[\|X - X(:,J)X(J,J)^{-1}X(J,:)\|_*] \le (r+1) \|(I - V V^T) X (I - V V^T)\|_*.
\]
\end{itemize}

The Householder-based implementation of ARP (Algorithm~\ref{alg:arp}) requires \(\mathcal{O}(n_2 r^2)\) operations. This is significantly cheaper than the deterministic derandomized version (Osinsky's algorithm), which requires \(\mathcal{O}(n_1 n_2 r)\) operations and full access to \(X\). The lower cost makes ARP suitable for large-scale problems where \(X\) cannot be accessed repeatedly.

A deterministic version is obtained by replacing the random sampling step with a greedy selection:
\[
j_k \in \arg\min_j \frac{\|\widetilde{X}_{k-1}(:,j)\|_2^2}{\|W_{k-1}(j,k:r)\|_2^2},
\]
where \(\widetilde{X}_k\) is a residual matrix. This recovers Osinsky's algorithm \cite{osinsky2023close} and guarantees
\[
\|X - X(:,J)V(J,:)^{-T}V^T\|_F^2 \le (r+1)\|X - X V V^T\|_F^2
\]
deterministically, at the cost of higher computational complexity.

For the SPSD case, the authors derive a novel deterministic algorithm (Algorithm~5.1) that avoids explicit computation of a square root factor \(B\) by operating directly on \(X\) while maintaining \(\mathcal{O}(n_2 r^2)\) complexity plus the cost of forming \(XV\).

\begin{remark}
The proof of Theorem~\ref{th_cor_kres} also applies to complex matrices after replacing transposes by conjugate transposes and Euclidean inner products by Hermitian inner products. We use this complex version for ARP-T-CUR in the Fourier domain.  
\end{remark}

\section{
\texorpdfstring
{Adaptive randomized pivoting in the t-product framework}
{Adaptive randomized pivoting in the t-product framework}
}\label{sec:tsvd}
This section presents two extensions of ARP to tensors in the t-product framework. ARP-T-CUR applies matrix ARP-cross at the nonredundant Fourier frequencies and couples conjugate pairs. T-ARP instead selects common lateral and horizontal slices, using the same indices at every Fourier frequency.

\color{black}

\subsection{\texorpdfstring{Fourier-slicewise ARP-T-CUR}{Fourier-slicewise ARP-T-CUR}}
\label{subsec:arp_t_cur}

We first describe a Fourier-slicewise extension of ARP. After applying the FFT along the third mode, the t-product decouples into matrix products on the frontal slices. Matrix ARP-cross can therefore be applied at each nonredundant frequency. The resulting method is a Fourier-domain cross approximation. Its indices may depend on frequency, so it does not in general select common lateral and horizontal slices in the original tensor domain.

\begin{algorithm}[H]
\caption{Fourier-slicewise ARP-T-CUR}
\label{alg:arp-tsvd}
\begin{algorithmic}[1]
\Input Tensor \(\mathcal X\in\mathbb R^{n_1\times n_2\times n_3}\), target rank \(r\), and orthonormal bases \(V^{(\ell)}\in\mathbb C^{n_2\times r}\) for the nonredundant Fourier frequencies, chosen compatibly with conjugate symmetry.
\Output Approximation \(\mathcal X_{\mathrm{TCUR}}\in\mathbb R^{n_1\times n_2\times n_3}\).
\State Compute \(\widehat{\mathcal X}=\texttt{fft}(\mathcal X,[],3)\).
\State Set \(L=\lfloor n_3/2\rfloor+1\).
\For{\(\ell=1,2,\ldots,L\)}
    \State Apply matrix ARP-cross to \((\widehat X^{(\ell)},V^{(\ell)})\), obtaining \(I_\ell,J_\ell\), each of cardinality \(r\).
    \State Form \(\widehat C^{(\ell)}=\widehat X^{(\ell)}(:,J_\ell)\), \(\widehat R^{(\ell)}=\widehat X^{(\ell)}(I_\ell,:)\).
    \State Assuming \(\widehat X^{(\ell)}(I_\ell,J_\ell)\) is nonsingular, set
    \[
    \widehat X_{\mathrm{TCUR}}^{(\ell)}
    =\widehat C^{(\ell)}\bigl(\widehat X^{(\ell)}(I_\ell,J_\ell)\bigr)^{-1}\widehat R^{(\ell)}.
    \]
    \State Let \(\ell'=n_3-\ell+2\). If \(\ell'\neq \ell\) and \(2\leq \ell'\leq n_3\), set \(\widehat X_{\mathrm{TCUR}}^{(\ell')}=\overline{\widehat X_{\mathrm{TCUR}}^{(\ell)}}\).
\EndFor
\State Set \(\mathcal X_{\mathrm{TCUR}}=\operatorname{real}(\texttt{ifft}(\widehat{\mathcal X}_{\mathrm{TCUR}},[],3))\).
\State \Return \(\mathcal X_{\mathrm{TCUR}}\).
\end{algorithmic}
\end{algorithm}

The bases may be obtained by an exact SVD, a randomized range finder, or another row-space approximation. For paired frequencies \(\ell' = n_3-\ell+2\), we use \(V^{(\ell')}=\overline{V^{(\ell)}}\) and the same random choices. At the self-conjugate DC frequency, and also at the Nyquist frequency when \(n_3\) is even, the basis and slicewise approximation are chosen real.

\begin{theorem}[Error bound for Fourier-slicewise ARP-T-CUR]
\label{thm:arp-tsvd-error}
Let \(\mathcal X\in\mathbb R^{n_1\times n_2\times n_3}\), and let \(V^{(\ell)}\in\mathbb C^{n_2\times r}\) be orthonormal bases satisfying the conjugate-pair convention above. Assume that, for every nonredundant frequency, the selected column matrix has rank \(r\) and the ARP-cross intersection \(\widehat X^{(\ell)}(I_\ell,J_\ell)\) is nonsingular almost surely. Then the output of Algorithm~\ref{alg:arp-tsvd} satisfies
\[
\mathbb E\|\mathcal X-\mathcal X_{\mathrm{TCUR}}\|_F^2
\leq
\frac{(r+1)^2}{n_3}\sum_{\ell=1}^{n_3}
\|\widehat X^{(\ell)}-\widehat X^{(\ell)}V^{(\ell)}(V^{(\ell)})^H\|_F^2.
\]
If \(V^{(\ell)}\) contains the leading \(r\) right singular vectors of \(\widehat X^{(\ell)}\) for every frequency, then
\[
\mathbb E\|\mathcal X-\mathcal X_{\mathrm{TCUR}}\|_F^2
\leq (r+1)^2\|\mathcal X-\mathcal X_r\|_F^2.
\]
\end{theorem}

\begin{proof}
For each nonredundant frequency, the approximation is exactly the matrix ARP-cross approximation. Hence the matrix result gives
\[
\mathbb E\|\widehat X^{(\ell)}-\widehat X_{\mathrm{TCUR}}^{(\ell)}\|_F^2
\leq (r+1)^2\|\widehat X^{(\ell)}-\widehat X^{(\ell)}V^{(\ell)}(V^{(\ell)})^H\|_F^2.
\]
The same estimate holds for its conjugate partner because the two approximation errors have the same Frobenius norm. Summing over all frequencies and using Parseval's identity proves the first claim. For leading right singular vectors, the sum on the right is the Fourier representation of the truncated T-SVD error, which proves the second claim.
\end{proof}

Thus, the factor \((r+1)^2\) concerns the expected \emph{squared} Frobenius error. Equivalently, Jensen's inequality gives an \((r+1)\)-factor for the expected Frobenius norm.

\subsection{\texorpdfstring{Common-index T-ARP}{Common-index T-ARP}}

\color{black}
We now describe the common-index extension. In contrast with ARP-T-CUR, which works at the nonredundant Fourier frequencies, T-ARP selects the same indices for the whole tensor. More precisely, for a tensor
\[
\mathcal X \in \mathbb R^{n_1\times n_2\times n_3},
\]
we aim to select lateral slices indexed by
\[
J=(j_1,j_2,\ldots,j_r)\subset \{1,2,\ldots,n_2\},
\]
and horizontal slices indexed by
\[
I=(i_1,i_2,\ldots,i_r)\subset \{1,2,\ldots,n_1\}.
\]
The corresponding tensor cross approximation has the form
\[
\mathcal X
\approx
\mathcal X(:,J,:)
*
\mathcal X(I,J,:)^\dagger
*
\mathcal X(I,:,:),
\]
where the pseudoinverse is taken in the t-product sense. If the intersection tensor \(\mathcal X(I,J,:)\) is invertible, then \(\mathcal X(I,J,:)^\dagger\) can be replaced by \(\mathcal X(I,J,:)^{-1}\). We refer to this common-index extension as T-ARP; see Figure~\ref{fig:TCUR} for an illustration.
\color{black}

\begin{figure}[t]
    \centering
\includegraphics[width=0.70\textwidth]{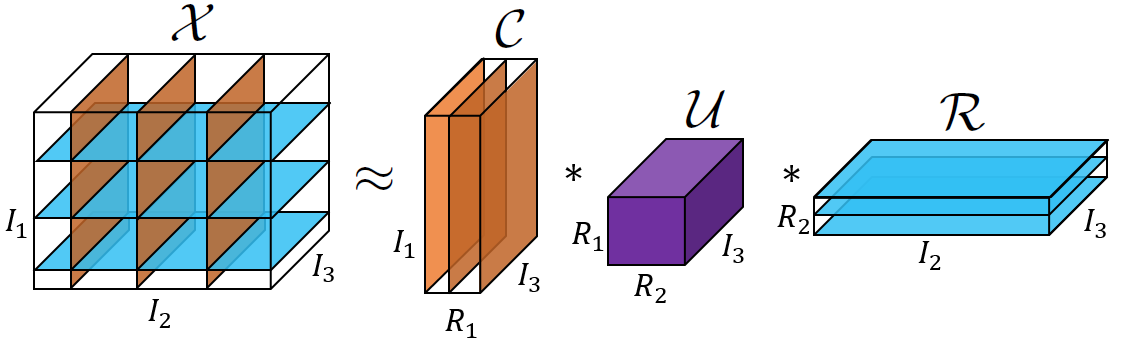}
    \caption{Tensor approximation based on sampling lateral and horizontal slices.}
    \label{fig:TCUR}
\end{figure}

Similar to the matrix case, Algorithm~\ref{alg:tarp_prototype} can be used in two stages to build a tensor cross approximation. We first use a right basis to select lateral slices. We then compute an orthonormal basis of the selected lateral slices and apply T-ARP again to select horizontal slices. This gives the approximation described in Algorithm~\ref{alg:tcross_compact}.

For a fixed index set \(J\), let
\[
\mathcal C
=
\mathcal X(:,J,:),
\]
and let \(\mathcal Q_J\) be an orthonormal basis of \(\operatorname{span}_t(\mathcal C)\), that is,
\[
\mathcal Q_J^\top * \mathcal Q_J
=
\mathcal I_r.
\]
The corresponding orthogonal projection error is
\[
\left\|
\mathcal X
-
\mathcal Q_J * \mathcal Q_J^\top * \mathcal X
\right\|_F^2.
\]
By the optimality of the truncated T-SVD, this error is always bounded below by the best tubal-rank-\(r\) approximation error:
\[
\left\|
\mathcal X
-
\mathcal X_r
\right\|_F^2
=
\sum_{k=r+1}^{\min(n_1,n_2)}
\left\|
\mathcal S(k,k,:)
\right\|_2^2
\leq
\left\|
\mathcal X
-
\mathcal Q_J * \mathcal Q_J^\top * \mathcal X
\right\|_F^2.
\]

\begin{algorithm}[H]
\caption{Tubal Adaptive Randomized Pivoting (T-ARP)}
\label{alg:tarp_prototype}
\begin{algorithmic}[1]
\Input A tensor \(\mathcal V\in\mathbb R^{n\times r\times p}\) satisfying \(\mathcal V^\top*\mathcal V=\mathcal I_r\).
\Output An index set \(J=(j_1,j_2,\ldots,j_r)\).
\State Initialize \(J_0=()\) and \(\mathcal W^{(0)}=\mathcal V\).
\For{\(k=1,2,\ldots,r\)}
    \State Compute \(\omega_j^{(k)}=\|\mathcal W^{(k-1)}(j,:,:)\|_F^2\), \(j=1,2,\ldots,n\).
    \State Set \(p_j^{(k)}=\omega_j^{(k)}/\sum_{i=1}^n\omega_i^{(k)}\), and sample \(j_k\) according to these probabilities.
    \State Set \(J_k=(J_{k-1},j_k)\) and \(\mathcal Z_k=\mathcal W^{(k-1)}(j_k,:,:)\).
    \State Update
    \[
    \mathcal W^{(k)}=\mathcal W^{(k-1)}*\bigl(\mathcal I_r-\mathcal Z_k^\dagger*\mathcal Z_k\bigr).
    \]
\EndFor
\State \Return \(J=J_r\).
\end{algorithmic}
\end{algorithm}

\begin{algorithm}[H]
\caption{T-Cross approximation based on T-ARP}
\label{alg:tcross_compact}
\begin{algorithmic}[1]
\Input A tensor \(\mathcal X\in\mathbb R^{n_1\times n_2\times n_3}\), target rank \(r\leq\min(n_1,n_2)\), and a right basis \(\mathcal V_R\in\mathbb R^{n_2\times r\times n_3}\) with \(\mathcal V_R^\top*\mathcal V_R=\mathcal I_r\).
\Output A tensor cross approximation \(\widehat{\mathcal X}\).
\State \(J\leftarrow\mathrm{T\text{-}ARP}(\mathcal V_R,r)\), and set \(\mathcal C=\mathcal X(:,J,:)\).
\State Compute \(\mathcal C=\mathcal Q_J*\mathcal R_J\), with \(\mathcal Q_J^\top*\mathcal Q_J=\mathcal I_r\).
\State \(I\leftarrow\mathrm{T\text{-}ARP}(\mathcal Q_J,r)\), and set \(\mathcal R=\mathcal X(I,:,:)\).
\State Set \(\mathcal U=\mathcal X(I,J,:)^\dagger\) and \(\widehat{\mathcal X}=\mathcal C*\mathcal U*\mathcal R\).
\State \Return \(\widehat{\mathcal X}\).
\end{algorithmic}
\end{algorithm}

The recursive update is convenient both for the analysis and for an FFT implementation. A Householder implementation may be used to represent the same sequence of residual subspaces more stably; the theory below is written for the exact algebraic update.

The next lemma gives the compact form of the recursive update.

\begin{lemma}[Compact representation of the T-ARP updates]
\label{lem:tarp-compact}
Let \(J_k=(j_1,\ldots,j_k)\), \(\mathcal E_{J_k}=\mathcal I_n(:,J_k,:)\), and \(\mathcal A_k=\mathcal E_{J_k}^\top*\mathcal V\). Then the iterates of Algorithm~\ref{alg:tarp_prototype} satisfy
\[
\mathcal W^{(k)}=\mathcal V*\bigl(\mathcal I_r-\mathcal A_k^\dagger*\mathcal A_k\bigr),\qquad k=0,1,\ldots,r.
\]
Let \(\widehat W^{(k,\ell)}\) denote the \(\ell\)-th frontal slice of the Fourier transform of \(\mathcal W^{(k)}\). Equivalently, for every frequency \(\ell\),
\[
\widehat W^{(k,\ell)}=\widehat V^{(\ell)}\left[I_r-\bigl(E_{J_k}^H\widehat V^{(\ell)}\bigr)^\dagger\bigl(E_{J_k}^H\widehat V^{(\ell)}\bigr)\right].
\]
\end{lemma}

\begin{proof}
Fix a frequency \(\ell\) and write \(V=\widehat V^{(\ell)}\). For \(k=0\), the set \(J_0\) is empty, \(A_0\) has no rows, and \(I-A_0^\dagger A_0=I\), so the claimed identity reduces to \(W^{(0)}=V\).

Assume now that \(W^{(k-1)}=VP_{k-1}\), where
\[
A_{k-1}=E_{J_{k-1}}^HV,
\qquad
P_{k-1}=I-A_{k-1}^\dagger A_{k-1}.
\]
The matrix \(P_{k-1}\) is the orthogonal projector onto \(\ker(A_{k-1})\). Let
\[
b=e_{j_k}^HV,
\qquad
z=bP_{k-1}=e_{j_k}^HW^{(k-1)}.
\]
Since \(A_k=\begin{bmatrix}A_{k-1}\\ b\end{bmatrix}\), a vector \(x\) belongs to \(\ker(A_k)\) if and only if \(x\in\ker(A_{k-1})=\operatorname{range}(P_{k-1})\) and \(bx=0\). For \(x=P_{k-1}x\), the latter condition is equivalent to \(zx=0\). Therefore
\[
\ker(A_k)=\operatorname{range}(P_{k-1})\cap\ker(z).
\]

If \(z\neq0\), then \(zP_{k-1}=z\), and the orthogonal projector onto this intersection is
\[
P_k
=
P_{k-1}-\frac{z^Hz}{\|z\|_2^2}.
\]
Indeed, the rank-one projector \(z^Hz/\|z\|_2^2\) has range contained in \(\operatorname{range}(P_{k-1})\), and subtracting it removes precisely the direction orthogonal to \(\ker(z)\) inside that subspace. If \(z=0\), the new row adds no constraint and \(P_k=P_{k-1}\). In both cases, \(P_k\) is the orthogonal projector onto \(\ker(A_k)\), hence
\[
P_k=I-A_k^\dagger A_k.
\]

Finally, when \(z\neq0\), the slicewise update of Algorithm~\ref{alg:tarp_prototype} gives

\[
\begin{aligned}
W^{(k)}
&=W^{(k-1)}\left(I-\frac{z^Hz}{\|z\|_2^2}\right)\\
&=VP_{k-1}\left(I-\frac{z^Hz}{\|z\|_2^2}\right)\\
&=V\left(P_{k-1}-\frac{z^Hz}{\|z\|_2^2}\right)
=VP_k.
\end{aligned}
\]

The case \(z=0\) gives the same conclusion with \(P_k=P_{k-1}\). Thus \(W^{(k)}=V(I-A_k^\dagger A_k)\) at every frequency. Applying the inverse FFT proves the tensor identity.
\end{proof}

\begin{corollary}[Selected slices are annihilated]
\label{cor:selected-slices-zero}
For every \(k=0,1,\ldots,r\),
\[
\mathcal E_{J_k}^\top * \mathcal W^{(k)}
=
0.
\]
Equivalently,
\[
\mathcal W^{(k)}(j,:,:)
=
0
\qquad
\text{for every } j\in J_k .
\]
\end{corollary}

\begin{proof}
Using the notation of Lemma~\ref{lem:tarp-compact}, we have
\[
\mathcal E_{J_k}^\top * \mathcal W^{(k)}
=
\mathcal A_k
*
\left(
\mathcal I_r
-
\mathcal A_k^\dagger
*
\mathcal A_k
\right).
\]
For any tensor \(\mathcal A_k\), the Moore--Penrose identity gives
\[
\mathcal A_k
*
\mathcal A_k^\dagger
*
\mathcal A_k
=
\mathcal A_k.
\]
Therefore,
\[
\mathcal A_k
*
\left(
\mathcal I_r
-
\mathcal A_k^\dagger
*
\mathcal A_k
\right)
=
0.
\]
This proves the claim.
\end{proof}

\begin{remark}[Normalization of the sampling probabilities]
The identity
\[
\left\|
\mathcal W^{(k)}
\right\|_F^2
=
r-k
\]
does not hold automatically for the common-index tensor algorithm. In the Fourier domain,
\[
\left\|
\mathcal W^{(k)}
\right\|_F^2
=
r
-
\frac{1}{p}
\sum_{\ell=1}^{p}
\operatorname{rank}
\left(
E_{J_k}^H
\widehat V^{(\ell)}
\right).
\]
Hence \(\left\|\mathcal W^{(k)}\right\|_F^2=r-k\) only under the additional condition that
\[
\operatorname{rank}
\left(
E_{J_k}^H
\widehat V^{(\ell)}
\right)
=
k
\qquad
\text{for every } \ell=1,2,\ldots,p.
\]
For this reason, the probabilities in the common-index T-ARP algorithm should be defined by explicit normalization:
\[
p_j^{(k)}
=
\frac{
\left\|
\mathcal W^{(k-1)}(j,:,:)
\right\|_F^2
}{
\left\|
\mathcal W^{(k-1)}
\right\|_F^2
},
\qquad
j=1,2,\ldots,n.
\]
Therefore Algorithm~\ref{alg:tarp_prototype} uses the explicitly normalized probabilities above. The denominator \(r-k+1\) is valid only when the frequency-wise rank condition holds.
\end{remark}

\begin{remark}[Computational interpretation]
Lemma~\ref{lem:tarp-compact} shows that \(\mathcal W^{(k)}\) is the part of \(\mathcal V\) that remains orthogonal to the selected tensor rows. In practice, the same residual subspaces can be maintained frequency by frequency using stable Householder or QR updates. We do not use an exact flop count here because it depends on how conjugate frequencies and the compact factors are stored.
\end{remark}

The indices \(J=(j_1,j_2,\ldots,j_r)\) selected by T-ARP can also be used to define an oblique interpolation operator. Let
\[
\mathcal E_J
=
\mathcal I_n(:,J,:)
\in
\mathbb R^{n\times r\times p},
\]
and assume that
\[
\mathcal V^\top * \mathcal E_J
\in
\mathbb R^{r\times r\times p}
\]
is invertible in the t-product sense. We define
\[
\widetilde{\Pi}_J
=
\mathcal I_n
-
\mathcal E_J
*
\left(
\mathcal V^\top
*
\mathcal E_J
\right)^{-1}
*
\mathcal V^\top .
\]
Then
\[
\mathcal X
*
\widetilde{\Pi}_J
=
\mathcal X
-
\mathcal X(:,J,:)
*
\left(
\mathcal V^\top
*
\mathcal E_J
\right)^{-1}
*
\mathcal V^\top .
\]
Since
\[
\mathcal V(J,:,:)^\top
=
\mathcal V^\top
*
\mathcal E_J,
\]
we use the notation \(\mathcal A^{-T}:=(\mathcal A^\top)^{-1}\) for an invertible square tensor \(\mathcal A\). Thus, this can also be written as
\[
\mathcal X
*
\widetilde{\Pi}_J
=
\mathcal X
-
\mathcal X(:,J,:)
*
\mathcal V(J,:,:)^{-T}
*
\mathcal V^\top .
\]

\begin{lemma}[Algebraic properties of the tensor oblique projector]
\label{lem:tensor-oblique}
Let \(\mathcal V\in \mathbb R^{n\times r\times p}\) satisfy
\[
\mathcal V^\top * \mathcal V
=
\mathcal I_r,
\]
and let \(J\) be an index set of cardinality \(r\) such that
\[
\mathcal V^\top * \mathcal E_J
\]
is invertible in the t-product sense. Then the tensor
\[
\widetilde{\Pi}_J
=
\mathcal I_n
-
\mathcal E_J
*
\left(
\mathcal V^\top
*
\mathcal E_J
\right)^{-1}
*
\mathcal V^\top
\]
satisfies the following properties:
\begin{enumerate}
    \item \(\widetilde{\Pi}_J\) is a projector:
    \[
    \widetilde{\Pi}_J * \widetilde{\Pi}_J
    =
    \widetilde{\Pi}_J .
    \]

    \item The cancellation is on the left:
    \[
    \mathcal V^\top * \widetilde{\Pi}_J
    =
    0.
    \]

    \item The selected coordinates are interpolated:
    \[
    \widetilde{\Pi}_J * \mathcal E_J
    =
    0.
    \]
    Consequently,
    \[
    \left(
    \mathcal X
    *
    \widetilde{\Pi}_J
    \right)(:,J,:)
    =
    0.
    \]

    \item We have
    \[
    \left(
    \mathcal I_n
    -
    \mathcal V
    *
    \mathcal V^\top
    \right)
    *
    \widetilde{\Pi}_J
    =
    \widetilde{\Pi}_J .
    \]
\end{enumerate}
\end{lemma}

\begin{proof}
All identities can be verified in the Fourier domain. Fix a frequency \(\ell\), and write
\[
V=\widehat V^{(\ell)},
\qquad
E=E_J,
\qquad
\widetilde{\Pi}
=
I_n
-
E
\left(
V^H E
\right)^{-1}
V^H .
\]

First,
\[
\widetilde{\Pi}^2
=
\left(
I_n
-
E
\left(
V^H E
\right)^{-1}
V^H
\right)^2 .
\]
Expanding gives
\[
\widetilde{\Pi}^2
=
I_n
-
2E
\left(
V^H E
\right)^{-1}
V^H
+
E
\left(
V^H E
\right)^{-1}
V^H
E
\left(
V^H E
\right)^{-1}
V^H .
\]
Since
\[
V^H E
\left(
V^H E
\right)^{-1}
=
I_r,
\]
the last term is equal to
\[
E
\left(
V^H E
\right)^{-1}
V^H .
\]
Hence
\[
\widetilde{\Pi}^2
=
\widetilde{\Pi}.
\]

Second,
\[
V^H\widetilde{\Pi}
=
V^H
-
V^H E
\left(
V^H E
\right)^{-1}
V^H
=
V^H
-
V^H
=
0.
\]

Third,
\[
\widetilde{\Pi}E
=
E
-
E
\left(
V^H E
\right)^{-1}
V^H E
=
E
-
E
=
0.
\]
Therefore,
\[
\left(
\mathcal X
*
\widetilde{\Pi}_J
\right)(:,J,:)
=
\mathcal X
*
\widetilde{\Pi}_J
*
\mathcal E_J
=
0.
\]

Finally,
\[
\left(
I_n
-
VV^H
\right)
\widetilde{\Pi}
=
I_n
-
VV^H
-
E
\left(
V^H E
\right)^{-1}
V^H
+
VV^H E
\left(
V^H E
\right)^{-1}
V^H .
\]
Using again
\[
V^H E
\left(
V^H E
\right)^{-1}
=
I_r,
\]
the last term becomes \(VV^H\), which cancels the term \(-VV^H\). Thus
\[
\left(
I_n
-
VV^H
\right)
\widetilde{\Pi}
=
I_n
-
E
\left(
V^H E
\right)^{-1}
V^H
=
\widetilde{\Pi}.
\]
Since the same identities hold for every frequency \(\ell\), applying the inverse FFT gives the tensor identities. 
\end{proof}

\color{black}

\begin{lemma}[Elementary oblique updates and final factorization]
\label{lem:elementary-oblique}
Fix selected indices \(J=(j_1,j_2,\ldots,j_r)\). At frequency \(\ell\), set \(z_k^{(\ell)}=e_{j_k}^H\widehat W^{(k-1,\ell)}\), and define
\[
\widehat{\widetilde\Pi}_k^{(\ell)}=
I_n-e_{j_k}\frac{z_k^{(\ell)}(\widehat W^{(k-1,\ell)})^H}{\|z_k^{(\ell)}\|_2^2}
\]
when \(z_k^{(\ell)}\neq0\), and \(\widehat{\widetilde\Pi}_k^{(\ell)}=I_n\) otherwise. Then
\[
\widehat W^{(k,\ell)}=(\widehat{\widetilde\Pi}_k^{(\ell)})^H\widehat W^{(k-1,\ell)}.
\]
If \(E_J^H\widehat V^{(\ell)}\) is nonsingular, then
\[
\widehat{\widetilde\Pi}_1^{(\ell)}\cdots\widehat{\widetilde\Pi}_r^{(\ell)}
=I_n-E_J\bigl((\widehat V^{(\ell)})^HE_J\bigr)^{-1}(\widehat V^{(\ell)})^H.
\]
Moreover, if \(\widehat R^{(0,\ell)}=\widehat X^{(\ell)}[I-\widehat V^{(\ell)}(\widehat V^{(\ell)})^H]\) and \(\widehat R^{(k,\ell)}=\widehat R^{(k-1,\ell)}\widehat{\widetilde\Pi}_k^{(\ell)}\), then
\[
\widehat R^{(k,\ell)}\widehat W^{(k,\ell)}=0
\]
and, whenever \(z_k^{(\ell)}\neq0\),
\[
\|\widehat R^{(k,\ell)}\|_F^2
=\|\widehat R^{(k-1,\ell)}\|_F^2
+\frac{\|\widehat R^{(k-1,\ell)}(:,j_k)\|_2^2}{\|z_k^{(\ell)}\|_2^2}.
\]
\end{lemma}

\begin{proof}
Fix the frequency \(\ell\) and omit the superscript \((\ell)\) in this proof. Write \(W_{k-1}=\widehat W^{(k-1,\ell)}\), \(z_k=e_{j_k}^HW_{k-1}\), and
\[
\widetilde\Pi_k
=
I-e_{j_k}\frac{z_kW_{k-1}^H}{\|z_k\|_2^2}.
\]
Since \(W_{k-1}=VP_{k-1}\) with \(P_{k-1}\) an orthogonal projector,
\[
W_{k-1}^HW_{k-1}
=
P_{k-1}V^HVP_{k-1}
=
P_{k-1}.
\]
Moreover, \(z_kP_{k-1}=z_k\). Hence

\[
\begin{aligned}
\widetilde\Pi_k^H W_{k-1}
&=
W_{k-1}
-
W_{k-1}\frac{z_k^He_{j_k}^HW_{k-1}}{\|z_k\|_2^2}\\
&=
W_{k-1}
-
W_{k-1}\frac{z_k^Hz_k}{\|z_k\|_2^2}
=
W_k,
\end{aligned}
\]

which proves the first identity.

We next prove the invariant \(R_kW_k=0\), where \(R_k=\widehat R^{(k,\ell)}\). At \(k=0\),
\[
R_0W_0
=
\widehat X^{(\ell)}(I-VV^H)V
=
0.
\]
Assume \(R_{k-1}W_{k-1}=0\). Since \(W_k=W_{k-1}P_k\), we have
\[
W_{k-1}^HW_k
=
P_{k-1}P_k
=
P_k.
\]
Moreover, the update removes the direction \(z_k^H\), so \(z_kP_k=0\). Therefore
\[
\widetilde\Pi_kW_k
=
W_k-e_{j_k}\frac{z_kW_{k-1}^HW_k}{\|z_k\|_2^2}
=
W_k-e_{j_k}\frac{z_kP_k}{\|z_k\|_2^2}
=
W_k.
\]
Consequently,
\[
R_kW_k
=
R_{k-1}\widetilde\Pi_kW_k
=
R_{k-1}W_k
=
0.
\]

The residual increment is
\[
\Delta_k
:=
R_k-R_{k-1}
=
-R_{k-1}(:,j_k)\frac{z_kW_{k-1}^H}{\|z_k\|_2^2}.
\]
Using the Frobenius inner product and the invariant,

\[
\begin{aligned}
\langle R_{k-1},\Delta_k\rangle_F
&=
-\frac{1}{\|z_k\|_2^2}
\operatorname{tr}\!\left(
R_{k-1}^H R_{k-1}(:,j_k)z_kW_{k-1}^H
\right)\\
&=
-\frac{1}{\|z_k\|_2^2}
z_kW_{k-1}^H R_{k-1}^H R_{k-1}(:,j_k)
=0,
\end{aligned}
\]

because \(R_{k-1}W_{k-1}=0\) implies \(W_{k-1}^HR_{k-1}^H=0\). Furthermore,
\[
\|z_kW_{k-1}^H\|_2^2
=
z_kW_{k-1}^HW_{k-1}z_k^H
=
z_kP_{k-1}z_k^H
=
\|z_k\|_2^2.
\]
Therefore
\[
\|\Delta_k\|_F^2
=
\frac{\|R_{k-1}(:,j_k)\|_2^2}{\|z_k\|_2^2}.
\]
Since \(R_{k-1}\) and \(\Delta_k\) are Frobenius-orthogonal, the Pythagorean identity yields the stated one-step norm formula.

It remains to prove the final factorization. Let \(P=\widetilde\Pi_1\cdots\widetilde\Pi_r\). Each factor has the form \(I-e_{j_k}c_k^H\), so every column of \(I-P\) lies in \(\operatorname{span}(E_J)\). Hence \(P=I-E_JB\) for some \(B\). For \(s<k\), Corollary~\ref{cor:selected-slices-zero} gives \(e_{j_s}^HW_{k-1}=0\), and therefore \(\widetilde\Pi_ke_{j_s}=e_{j_s}\). On the other hand,
\[
\widetilde\Pi_ke_{j_k}
=
e_{j_k}\left(
1-\frac{z_kW_{k-1}^He_{j_k}}{\|z_k\|_2^2}
\right)
=
0,
\]
because \(W_{k-1}^He_{j_k}=z_k^H\). Thus, after multiplication by all factors, \(PE_J=0\). Moreover, repeated use of the first identity gives \(P^HV=W_r\). Final nonsingularity of \(E_J^HV\) implies \(W_r=0\), and hence \(V^HP=0\). Substituting \(P=I-E_JB\) gives
\[
V^H-V^HE_JB=0,
\qquad
B=(V^HE_J)^{-1}V^H.
\]
Therefore
\[
P
=
I-E_J(V^HE_J)^{-1}V^H,
\]
which is the claimed final oblique projector.
\end{proof}

The previous lemmas give the algebraic identities needed to describe the oblique interpolation residual. In particular, for a fixed index set \(J\), the tensor
\[
\widetilde{\Pi}_J
=
\mathcal I_n
-
\mathcal E_J
*
\left(
\mathcal V^\top
*
\mathcal E_J
\right)^{-1}
*
\mathcal V^\top
\]
satisfies
\[
\mathcal X * \widetilde{\Pi}_J
=
\mathcal X
-
\mathcal X(:,J,:)
*
\mathcal V(J,:,:)^{-T}
*
\mathcal V^\top.
\]
The remaining question is probabilistic: we need to understand the expected size of this residual when \(J\) is generated by the common-index T-ARP sampling rule.

\paragraph{Frequency alignment and rank growth}

Let \(J_{k-1}\) be a history with positive probability. For each frequency, define
\[
a_{j,k}^{(\ell)}=\|\widehat W^{(k-1,\ell)}(j,:)\|_2^2,
\qquad
q_{j,k}=\frac{\sum_{\nu=1}^p a_{j,k}^{(\nu)}}{\sum_{i=1}^n\sum_{\nu=1}^p a_{i,k}^{(\nu)}}.
\]
The probability \(q_{j,k}\) is used by common-index T-ARP, while \(p_{j,k}^{(\ell)}\) is the probability that matrix ARP would use at the single frequency \(\ell\). The frequency-wise ARP probability is
\[
p_{j,k}^{(\ell)}=\frac{a_{j,k}^{(\ell)}}{\sum_{i=1}^n a_{i,k}^{(\ell)}}.
\]

\begin{assumption}[Frequency alignment and rank growth]
\label{ass:tarp-alignment}
There are constants \(\eta_k\geq1\) such that, for every history \(J_{k-1}\) with positive probability, every frequency \(\ell\), and every index with \(a_{j,k}^{(\ell)}>0\),
\[
q_{j,k}\leq \eta_k p_{j,k}^{(\ell)}.
\]
Moreover, for every history \(J_k\) with positive probability and every frequency \(\ell\),
\[
\operatorname{rank}(E_{J_k}^H\widehat V^{(\ell)})=k,
\qquad k=0,1,\ldots,r.
\]
\end{assumption}

The rank-growth condition implies \(\sum_i a_{i,k}^{(\ell)}=r-k+1\) and, at \(k=r\), the nonsingularity of \(E_J^H\widehat V^{(\ell)}\). More explicitly, if \(q_{j,k}>0\), then the extended history \((J_{k-1},j)\) has positive probability. Applying rank growth to this history shows that \(a_{j,k}^{(\ell)}>0\) at every frequency. We call the distributions perfectly aligned at step \(k\) when \(q_{j,k}=p_{j,k}^{(\ell)}\) for all \(j,\ell\); this corresponds to \(\eta_k=1\).

\begin{theorem}[Expected error for common-index T-ARP under frequency alignment]
\label{thm:tarp-error}
Let \(\mathcal X\in\mathbb R^{m\times n\times p}\) and \(\mathcal V\in\mathbb R^{n\times r\times p}\) satisfy \(\mathcal V^\top*\mathcal V=\mathcal I_r\). Let \(J\) be returned by Algorithm~\ref{alg:tarp_prototype}, and suppose Assumption~\ref{ass:tarp-alignment} holds. Then
\[
\mathbb E\left\|\mathcal X-\mathcal X(:,J,:)*\mathcal V(J,:,:)^{-T}*\mathcal V^\top\right\|_F^2
\leq
\Gamma_\eta\left\|\mathcal X-\mathcal X*\mathcal V*\mathcal V^\top\right\|_F^2,
\]
where
\[
\Gamma_\eta=\prod_{k=1}^r\left(1+\frac{\eta_k}{r-k+1}\right).
\]
In the perfectly aligned case, \(\Gamma_\eta=r+1\).
\end{theorem}

\begin{proof}
Fix a frequency \(\ell\). Let
\[
\widehat R^{(0,\ell)}
=
\widehat X^{(\ell)}
\left[I-\widehat V^{(\ell)}(\widehat V^{(\ell)})^H\right],
\]
and define \(\widehat R^{(k,\ell)}\) recursively by the elementary projectors of Lemma~\ref{lem:elementary-oblique}. Conditionally on the history \(J_{k-1}\), an index \(j\) is selected with probability \(q_{j,k}\). Terms with \(q_{j,k}=0\) do not contribute. For every \(j\) with \(q_{j,k}>0\), the rank-growth assumption gives \(a_{j,k}^{(\ell)}>0\), and the one-step identity yields
\[
\mathbb E\!\left[\|\widehat R^{(k,\ell)}\|_F^2\mid J_{k-1}\right]
=
\|\widehat R^{(k-1,\ell)}\|_F^2
+
\sum_{\{j:\,q_{j,k}>0\}}
q_{j,k}
\frac{\|\widehat R^{(k-1,\ell)}(:,j)\|_2^2}{a_{j,k}^{(\ell)}}.
\]
For every term in this sum, the alignment condition gives
\[
\frac{q_{j,k}}{a_{j,k}^{(\ell)}}
\leq
\frac{\eta_k}{\sum_{i=1}^n a_{i,k}^{(\ell)}}
=
\frac{\eta_k}{r-k+1}.
\]
Consequently,

\[
\begin{aligned}
\mathbb E\!\left[\|\widehat R^{(k,\ell)}\|_F^2\mid J_{k-1}\right]
&\leq
\|\widehat R^{(k-1,\ell)}\|_F^2
+
\frac{\eta_k}{r-k+1}
\sum_{\{j:\,q_{j,k}>0\}}
\|\widehat R^{(k-1,\ell)}(:,j)\|_2^2\\
&\leq
\left(1+\frac{\eta_k}{r-k+1}\right)
\|\widehat R^{(k-1,\ell)}\|_F^2.
\end{aligned}
\]

For \(k=0,\ldots,r\), let \(\widehat{\mathcal R}^{(k)}\) be the tensor whose \(\ell\)-th frontal slice is \(\widehat R^{(k,\ell)}\), and set
\[
\mathcal R^{(k)}
=
\texttt{ifft}(\widehat{\mathcal R}^{(k)},[],3).
\]
Summing the preceding conditional estimate over the frequencies and using Parseval's identity gives
\[
\mathbb E\!\left[\|\mathcal R^{(k)}\|_F^2\mid J_{k-1}\right]
\leq
\left(1+\frac{\eta_k}{r-k+1}\right)
\|\mathcal R^{(k-1)}\|_F^2.
\]
Applying the tower property successively for \(k=1,\ldots,r\) yields
\[
\mathbb E\|\mathcal R^{(r)}\|_F^2
\leq
\Gamma_\eta\|\mathcal R^{(0)}\|_F^2.
\]
By Lemma~\ref{lem:elementary-oblique}, the product of the elementary factors is the final oblique projector. Hence
\[
\mathcal R^{(0)}
=
\mathcal X-\mathcal X*\mathcal V*\mathcal V^\top,
\qquad
\mathcal R^{(r)}
=
\mathcal X-\mathcal X(:,J,:)*\mathcal V(J,:,:)^{-T}*\mathcal V^\top.
\]
This proves the bound. If \(\eta_k=1\) for all \(k\), then
\[
\Gamma_\eta
=
\prod_{k=1}^r\left(1+\frac{1}{r-k+1}\right)
=
\prod_{s=1}^r\frac{s+1}{s}
=
r+1.
\]
\end{proof}

\paragraph{Consequences of Theorem~\ref{thm:tarp-error}}
We keep the notation \(\Gamma_\eta\) introduced in Theorem~\ref{thm:tarp-error}.

\begin{corollary}[Jensen's inequality bound]
\label{cor:tarp-jensen}
Under the assumptions of Theorem~\ref{thm:tarp-error}, we have
\[
\mathbb E
\left[
\left\|
\mathcal X
-
\mathcal X(:,J,:)
*
\mathcal V(J,:,:)^{-T}
*
\mathcal V^\top
\right\|_F
\right]
\leq
\sqrt{\Gamma_{\eta}}
\left\|
\mathcal X
-
\mathcal X
*
\mathcal V
*
\mathcal V^\top
\right\|_F .
\]
\end{corollary}

\begin{proof}
This follows directly from Jensen's inequality applied to the concave function \(t\mapsto \sqrt{t}\), together with Theorem~\ref{thm:tarp-error}.
\end{proof}

\begin{corollary}[Orthogonal projection bound]
\label{cor:tarp-orthogonal}
Let \(\mathcal Q_J\) be an orthonormal basis of \(\operatorname{span}_t(\mathcal X(:,J,:))\), and define the orthogonal projector \(\Pi_J=\mathcal Q_J*\mathcal Q_J^\top\).
Under the assumptions of Theorem~\ref{thm:tarp-error}, we have
\[
\mathbb E
\left[
\left\|
\mathcal X
-
\Pi_J * \mathcal X
\right\|_F^2
\right]
\leq
\Gamma_{\eta}
\left\|
\mathcal X
-
\mathcal X
*
\mathcal V
*
\mathcal V^\top
\right\|_F^2 .
\]
\end{corollary}

\begin{proof}
For a fixed index set \(J\), the orthogonal projection onto
\(\operatorname{span}_t(\mathcal X(:,J,:))\) gives the smallest Frobenius-norm error among all approximations whose lateral slices belong to this tensor subspace. In particular,
\[
\left\|
\mathcal X
-
\Pi_J * \mathcal X
\right\|_F
\leq
\left\|
\mathcal X
-
\mathcal X(:,J,:)
*
\mathcal V(J,:,:)^{-T}
*
\mathcal V^\top
\right\|_F .
\]
Squaring this inequality and taking expectations gives the result by Theorem~\ref{thm:tarp-error}.
\end{proof}

\subsection{Optimality for an exact row-space basis}

When \(\mathcal V\) contains the first \(r\) right singular tensors from the T-SVD of \(\mathcal X\), the residual term in Theorem~\ref{thm:tarp-error} is exactly the optimal tubal-rank-\(r\) approximation error. This gives the following consequence.

\begin{corollary}[Optimal t-CSSP bound under frequency alignment]
\label{cor:tcssp-optimal}
Let
\[
\mathcal X
=
\mathcal U
*
\mathcal S
*
\mathcal V_{\mathrm{opt}}^\top
\]
be the T-SVD of \(\mathcal X\), and let
\[
\mathcal V
=
\mathcal V_{\mathrm{opt}}(:,1:r,:)
\]
contain the first \(r\) right singular tensors. Assume that all hypotheses of Theorem~\ref{thm:tarp-error}, equivalently Assumption~\ref{ass:tarp-alignment}, hold for this basis. Then
\[
\mathbb E
\left[
\left\|
\mathcal X
-
\Pi_J * \mathcal X
\right\|_F^2
\right]
\leq
\Gamma_{\eta}
\sum_{k=r+1}^{\min(n_1,n_2)}
\left\|
\mathcal S(k,k,:)
\right\|_2^2 .
\]
In particular, if the leverage-score distributions are perfectly aligned across the Fourier slices, then
\[
\mathbb E
\left[
\left\|
\mathcal X
-
\Pi_J * \mathcal X
\right\|_F^2
\right]
\leq
(r+1)
\sum_{k=r+1}^{\min(n_1,n_2)}
\left\|
\mathcal S(k,k,:)
\right\|_2^2 .
\]
\end{corollary}

\begin{proof}
For
\[
\mathcal V
=
\mathcal V_{\mathrm{opt}}(:,1:r,:),
\]
we have
\[
\mathcal X
*
\mathcal V
*
\mathcal V^\top
=
\mathcal X_r,
\]
where \(\mathcal X_r\) is the truncated T-SVD approximation of tubal rank \(r\). Hence
\[
\left\|
\mathcal X
-
\mathcal X
*
\mathcal V
*
\mathcal V^\top
\right\|_F^2
=
\left\|
\mathcal X
-
\mathcal X_r
\right\|_F^2
=
\sum_{k=r+1}^{\min(n_1,n_2)}
\left\|
\mathcal S(k,k,:)
\right\|_2^2 .
\]
The result follows from Corollary~\ref{cor:tarp-orthogonal}.
\end{proof}

We now turn to tensor cross approximation. The first application of T-ARP selects the lateral slices, while the second application selects the horizontal slices. Since the second sampling step is itself a common-index T-ARP procedure, it may have its own frequency-alignment factor. We denote it by
\[
\Gamma_{\theta}
=
\prod_{k=1}^{r}
\left(
1+
\frac{\theta_k}{r-k+1}
\right),
\]
where the constants \(\theta_k\) control the frequency alignment for the second T-ARP call, conditionally on the selected lateral slices.

\begin{theorem}[T-Cross approximation error under frequency alignment]
\label{thm:tcross-error}
Let \(J\) be obtained by applying T-ARP to a right basis \(\mathcal V\), set \(\mathcal C=\mathcal X(:,J,:)\), and let \(\mathcal C=\mathcal Q_J*\mathcal R_J\) be a thin T-QR factorization with \(r\) orthonormal lateral slices. Apply T-ARP to \(\mathcal Q_J\) to obtain \(I\). Assume that \(\mathcal X(I,J,:)\) is t-invertible almost surely. Assume that the first T-ARP call satisfies Assumption~\ref{ass:tarp-alignment} with factor \(\Gamma_\eta\), and that, conditionally on every \(J\) in its support, the second call satisfies the analogous assumptions with a deterministic uniform factor \(\Gamma_\theta\). Then
\[
\mathbb E\left\|\mathcal X-\mathcal X(:,J,:)*\mathcal X(I,J,:)^{-1}*\mathcal X(I,:,:)\right\|_F^2
\leq \Gamma_\theta\Gamma_\eta\|\mathcal X-\mathcal X*\mathcal V*\mathcal V^\top\|_F^2.
\]
If both calls are perfectly aligned, the factor is \((r+1)^2\).
\end{theorem}

\begin{proof}
Fix an index set \(J\) in the support of the first T-ARP call. Since
\[
\mathcal X(I,J,:)
=
\mathcal E_I^\top*\mathcal X(:,J,:)
=
\mathcal E_I^\top*\mathcal Q_J*\mathcal R_J
\]
is invertible, both square factors \(\mathcal R_J\) and \(\mathcal E_I^\top*\mathcal Q_J\) are invertible. Using \((\mathcal A*\mathcal B)^{-1}=\mathcal B^{-1}*\mathcal A^{-1}\), we obtain

\[
\begin{aligned}
\mathcal X(:,J,:)*\mathcal X(I,J,:)^{-1}
&=
\mathcal Q_J*\mathcal R_J*
\bigl(\mathcal E_I^\top*\mathcal Q_J*\mathcal R_J\bigr)^{-1}\\
&=
\mathcal Q_J*\mathcal R_J*
\mathcal R_J^{-1}*
(\mathcal E_I^\top*\mathcal Q_J)^{-1}\\
&=
\mathcal Q_J*(\mathcal E_I^\top*\mathcal Q_J)^{-1}.
\end{aligned}
\]

Therefore the tensor cross approximation is exactly the row interpolant
\[
\mathcal Q_J*(\mathcal E_I^\top*\mathcal Q_J)^{-1}*
\mathcal E_I^\top*\mathcal X.
\]

Conditionally on \(J\), apply Theorem~\ref{thm:tarp-error} to \(\mathcal X^\top\) with right basis \(\mathcal Q_J\). After transposing the resulting approximation back, and using invariance of the Frobenius norm under tensor transposition, we get
\[
\mathbb E_{I\mid J}\!\left[
\left\|
\mathcal X-\mathcal X(:,J,:)*\mathcal X(I,J,:)^{-1}*\mathcal X(I,:,:)
\right\|_F^2
\right]
\leq
\Gamma_\theta
\|\mathcal X-\mathcal Q_J*\mathcal Q_J^\top*\mathcal X\|_F^2.
\]
For this fixed \(J\), the tensor \(\mathcal Q_J*\mathcal Q_J^\top*\mathcal X\) is the orthogonal projection of \(\mathcal X\) onto the t-span of \(\mathcal X(:,J,:)\). Hence
\[
\|\mathcal X-\mathcal Q_J*\mathcal Q_J^\top*\mathcal X\|_F^2
\leq
\left\|
\mathcal X-\mathcal X(:,J,:)*\mathcal V(J,:,:)^{-T}*\mathcal V^\top
\right\|_F^2.
\]
Using the law of total expectation and then Theorem~\ref{thm:tarp-error},

\[
\begin{aligned}
\mathbb E_{I,J}\!\left[
\left\|
\mathcal X-\mathcal X(:,J,:)*\mathcal X(I,J,:)^{-1}*\mathcal X(I,:,:)
\right\|_F^2
\right]
&=
\mathbb E_J\!\left[
\mathbb E_{I\mid J}\!\left[\cdot\right]
\right]\\
&\leq
\Gamma_\theta
\mathbb E_J
\|\mathcal X-\mathcal Q_J*\mathcal Q_J^\top*\mathcal X\|_F^2\\
&\leq
\Gamma_\theta\Gamma_\eta
\|\mathcal X-\mathcal X*\mathcal V*\mathcal V^\top\|_F^2.
\end{aligned}
\]

This proves the result.
\end{proof}

The Discrete Empirical Interpolation Method (DEIM) selects indices to approximate a function from a reduced basis. The same argument gives a tensor analogue with the frequency-alignment factor.

\begin{corollary}[t-DEIM error bound under frequency alignment]
\label{cor:tdeim-error}
Let \(\mathcal V\in\mathbb R^{n\times r\times p}\) satisfy \(\mathcal V^\top*\mathcal V=\mathcal I_r\), let \(\mathcal F\in\mathbb R^{n\times1\times p}\), and let \(I\) be returned by T-ARP applied to \(\mathcal V\). Assume all hypotheses of Theorem~\ref{thm:tarp-error}. Then
\[
\mathbb E\left\|\mathcal F-\mathcal V*(\mathcal E_I^\top*\mathcal V)^{-1}*\mathcal E_I^\top*\mathcal F\right\|_F^2
\leq \Gamma_\eta\|\mathcal F-\mathcal V*\mathcal V^\top*\mathcal F\|_F^2.
\]
In the perfectly aligned case, \(\Gamma_\eta=r+1\).
\end{corollary}

\begin{proof}
Apply Theorem~\ref{thm:tarp-error} to \(\mathcal X=\mathcal F^\top\in\mathbb R^{1\times n\times p}\) with right basis \(\mathcal V\). Its oblique approximation is
\[
\mathcal F^\top(:,I,:)*\mathcal V(I,:,:)^{-T}*\mathcal V^\top.
\]
Since \(\mathcal F^\top(:,I,:)=\mathcal F^\top*\mathcal E_I\), tensor transposition gives
\[
\bigl(\mathcal F^\top(:,I,:)\bigr)^\top
=
\mathcal E_I^\top*\mathcal F.
\]
Moreover, \(\mathcal V(I,:,:)=\mathcal E_I^\top*\mathcal V\), and therefore
\[
\left(\mathcal V(I,:,:)^{-T}\right)^\top
=
(\mathcal E_I^\top*\mathcal V)^{-1}.
\]
Reversing the order of the factors under transposition, the transpose of the oblique approximation is
\[
\mathcal V*
(\mathcal E_I^\top*\mathcal V)^{-1}*
\mathcal E_I^\top*\mathcal F.
\]
Finally, the Frobenius norm is invariant under tensor transposition, so Theorem~\ref{thm:tarp-error} gives the stated bound.
\end{proof}

\section{Numerical experiments}\label{SEc:NE}

This section presents numerical experiments on synthetic tensors, Kodak images, and a video sequence. The Kodak and video experiments were run on a MacBook Air with an M3 chip and 8~GB of RAM, using Python 3.11.14 and JAX 0.9.0.1. The synthetic and theorem-validation experiments were run separately in Google Colab using NumPy. The implementations and result files are available in the repository and its ``experiments'' branch: \url{https://github.com/ah-haqqdod/T-ARP} and \url{https://github.com/ah-haqqdod/T-ARP/tree/experiments}.

All experiments use relative error; visual data are also evaluated by PSNR and SSIM. The baselines are described in Section~\ref{subsec:decomposition_baselines}. In this section, lateral slices are called columns and horizontal slices are called rows. For the common-index methods, the target rank \(r\) is also the number of selected lateral and horizontal slices; ARP-T-CUR instead selects \(r\) rows and columns at each nonredundant Fourier frequency. In the practical benchmarks, T-ARP is run deterministically: at each step, the index with the largest current tensor leverage score is selected. The selected slices are evaluated through the two-sided projection reconstruction \(\mathcal C*\mathcal C^\dagger*\mathcal X*\mathcal R^\dagger*\mathcal R\). The separate experiment in Section~\ref{subsec:empirical-tarp-bound} uses the randomized T-ARP distribution and the exact one-sided interpolant from Theorem~\ref{thm:tarp-error}.

\begin{figure}[!hb]
  \centering
  \begin{tabular}{ccc}
    \begin{subfigure}[b]{0.27\textwidth}
      \includegraphics[width=\linewidth]{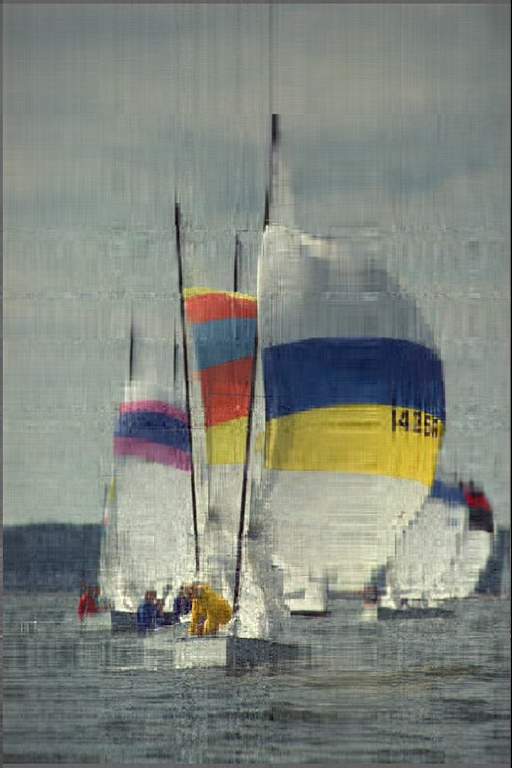}
      \caption{ARP-T-CUR}
    \end{subfigure} &
    \begin{subfigure}[b]{0.27\textwidth}
      \includegraphics[width=\linewidth]{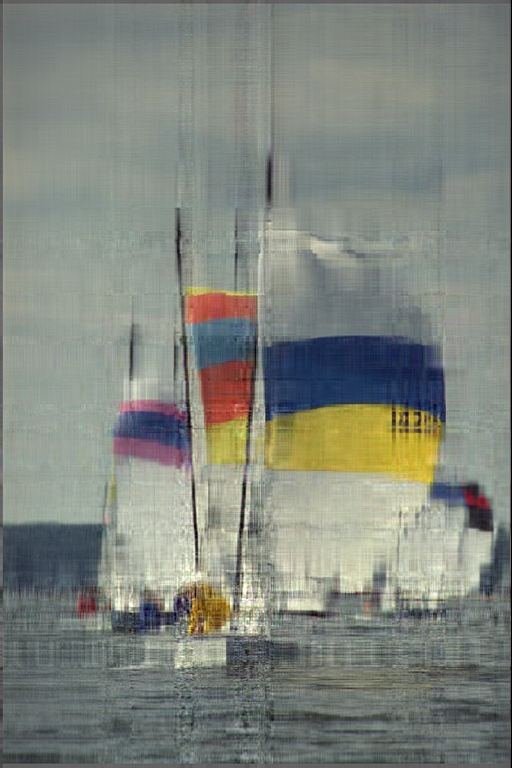}
      \caption{Uniform sampling}
    \end{subfigure} &
    \begin{subfigure}[b]{0.27\textwidth}
      \includegraphics[width=\linewidth]{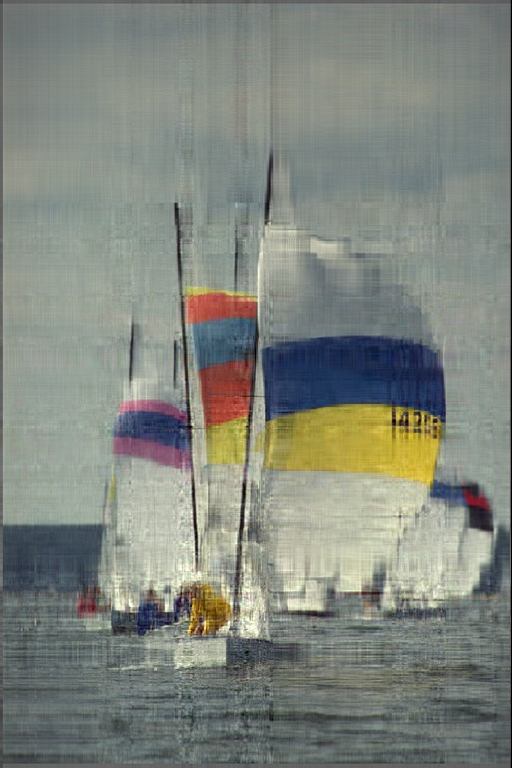}
      \caption{Lev. scores sampling}
    \end{subfigure} \\
    \begin{subfigure}[b]{0.27\textwidth}
      \includegraphics[width=\linewidth]{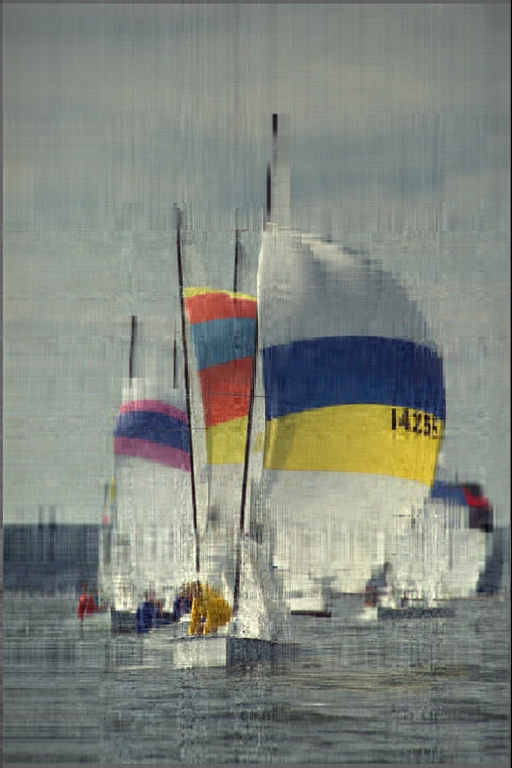}
      \caption{T-ARP}
    \end{subfigure} &
    \begin{subfigure}[b]{0.27\textwidth}
      \includegraphics[width=\linewidth]{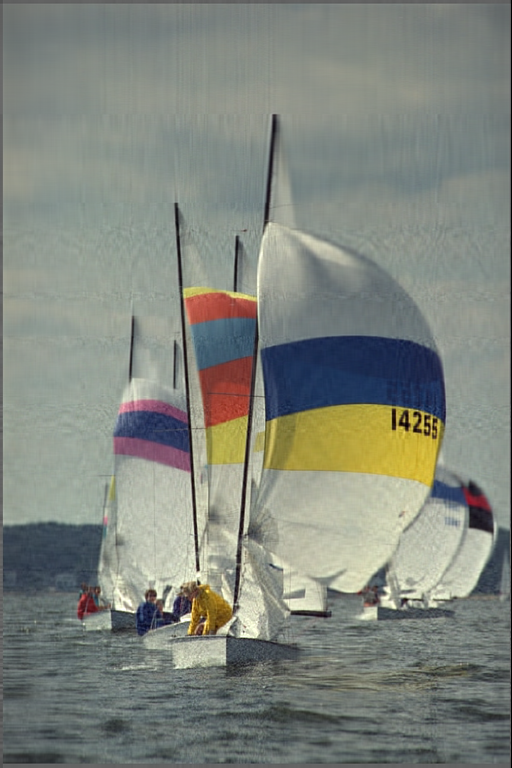}
      \caption{T-SVD}
    \end{subfigure} &
    \begin{subfigure}[b]{0.27\textwidth}
      \includegraphics[width=\linewidth]{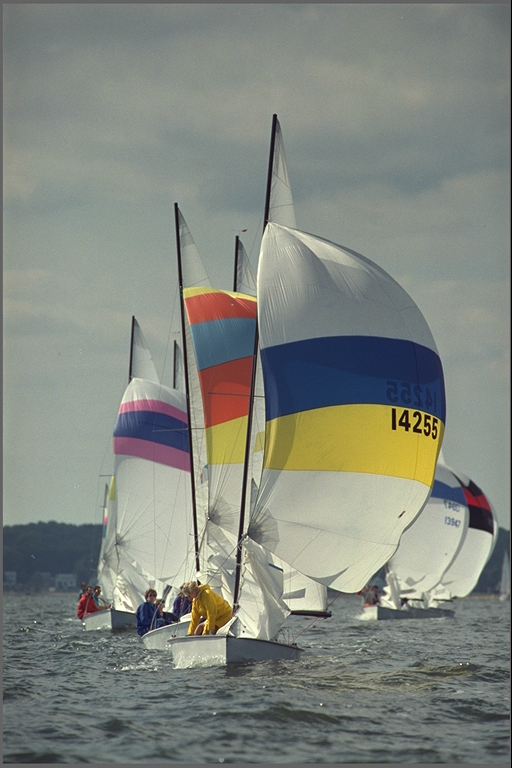}
      \caption{Original image}
    \end{subfigure} \\
  \end{tabular}
  \caption{Reconstructions of the Kodak image \textit{kodim9.png} with target rank \(r=60\), using \(\mathcal M \approx \mathcal C*\mathcal C^\dagger*\mathcal M*\mathcal R^\dagger*\mathcal R\). For the common-index methods, \(\mathcal C\in\mathbb R^{512\times60\times3}\) and \(\mathcal R\in\mathbb R^{60\times768\times3}\) contain the selected lateral and horizontal slices, respectively; ARP-T-CUR uses frequency-dependent selections. Additional reconstructions are available at \url{https://github.com/ah-haqqdod/T-ARP/tree/experiments/src/benchmarks/kodak/results/reconstructions}.} \label{fig:img_reconstruction_examples}
\end{figure}

\begin{figure}[!ht]
  \centering
  \begin{tabular}{cc}
    \begin{subfigure}[b]{0.3\textwidth}
      \includegraphics[width=\linewidth]{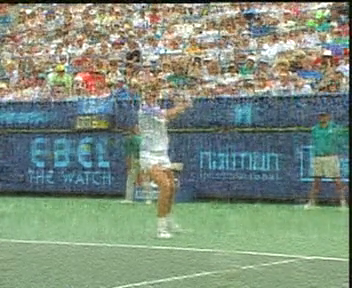}
      \caption{Uniform sampling}
    \end{subfigure} &
    \begin{subfigure}[b]{0.3\textwidth}
      \includegraphics[width=\linewidth]{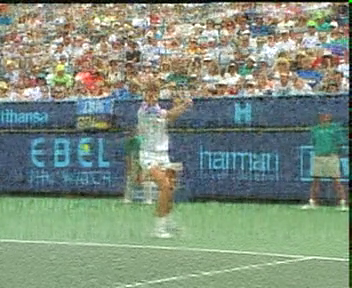}
      \caption{Lev. scores sampling}
    \end{subfigure} \\
    \begin{subfigure}[b]{0.3\textwidth}
      \includegraphics[width=\linewidth]{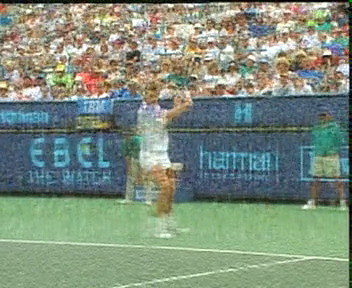}
      \caption{T-ARP}
    \end{subfigure} &
    \begin{subfigure}[b]{0.3\textwidth}
      \includegraphics[width=\linewidth]{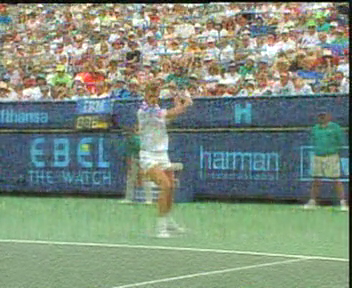}
      \caption{ARP T-CUR}
    \end{subfigure} \\
  \end{tabular}
  \caption{Frame~15 of the reconstructed YUV video \textit{stefan\_cif.yuv} with target rank \(r=60\), using \(\mathcal M \approx \mathcal C*\mathcal C^\dagger*\mathcal M*\mathcal R^\dagger*\mathcal R\). For the common-index methods, \(\mathcal C\in\mathbb R^{288\times60\times3\times90}\) and \(\mathcal R\in\mathbb R^{60\times352\times3\times90}\) contain the selected lateral and horizontal slices, respectively; ARP-T-CUR uses frequency-dependent selections. Additional reconstructions are available at \url{https://github.com/ah-haqqdod/T-ARP/tree/experiments/src/benchmarks/yuv/results}.}\label{fig:video_reconstruction_examples}
\end{figure}

\subsection{Metrics}
\label{subsec:exp_metrics}

\paragraph{Relative Error}
For two tensors \(\mathcal X,\mathcal Y\in\mathbb R^{m\times n\times p}\), the relative error is defined as
\[
\text{RE}(\mathcal X,\mathcal Y) = \frac{\|\mathcal X-\mathcal Y\|_F}{\|\mathcal X\|_F},
\]
where $\|\mathcal X\|_F$ is the Frobenius norm of $\mathcal X$.

\paragraph{Peak Signal-to-Noise Ratio (PSNR)}
For two images $\mathcal X$ and $\mathcal Y$ of size $m\times n\times 3$ with $L$ being the dynamic range of pixel values (e.g., $L=255$ for 8-bit images), \emph{PSNR} (in dB) is defined as
\[
\text{PSNR}(\mathcal X,\mathcal Y) = 10\log_{10}\left(\frac{L^2}{\text{MSE}(\mathcal X,\mathcal Y)}\right),
\]
where \(\text{MSE}(\mathcal X,\mathcal Y)=\|\mathcal X-\mathcal Y\|_F^2/(3mn)\) is the mean squared error; \emph{PSNR} is used to measure the quality of a reconstructed image compared to a reference.

\paragraph{Structural Similarity Index (SSIM)}
\emph{SSIM} is a perception-based model that considers changes in structural information, luminance, and contrast. The \emph{SSIM} between two images $\mathcal X$ and $\mathcal Y$ is computed as

$$\text{SSIM}(\mathcal X, \mathcal Y) = \frac{(2\mu_x\mu_y + C_1)(2\sigma_{xy} + C_2)}{(\mu_x^2 + \mu_y^2 + C_1)(\sigma_x^2 + \sigma_y^2 + C_2)},$$
where $\mu_x$, $\mu_y$ are the mean intensities, $\sigma_x^2$, $\sigma_y^2$ are the variances, and $\sigma_{xy}$ is the covariance. The constants $C_1 = (K_1L)^2$ and $C_2 = (K_2L)^2$ are used to avoid instability when the denominators are close to zero; typically $K_1=0.01$, $K_2=0.03$, and $L$ is the dynamic range of pixel values. Because the Kodak images are scaled to \([0,1]\), we use \(L=1\) for both PSNR and SSIM. 

\subsection{Decomposition Baselines}
\label{subsec:decomposition_baselines}

We use three slice-sampling baselines. Uniform slice sampling selects index sets \(J\subseteq\{1,\ldots,n\}\) and \(I\subseteq\{1,\ldots,m\}\), both of cardinality \(r\), uniformly without replacement. For \(\mathcal X\in\mathbb R^{m\times n\times p}\), it forms \(\mathcal C=\mathcal X(:,J,:)\in\mathbb R^{m\times r\times p}\) and \(\mathcal R=\mathcal X(I,:,:)\in\mathbb R^{r\times n\times p}\).

Length-squared slice sampling uses probabilities proportional to the squared Frobenius norms of the slices. For lateral slices,
\[
p_j=\frac{\|\mathcal X(:,j,:)\|_F^2}{\sum_{i=1}^n\|\mathcal X(:,i,:)\|_F^2},
\]
and \(r\) indices are sampled without replacement. Horizontal slices are selected in the same way from \(\mathcal X^\top\).

Leverage-score slice sampling uses the initial tensor leverage scores without the adaptive updates of T-ARP. It is described in Algorithms~\ref{alg:lev_scores_sampling} and~\ref{alg:tcross_lev_scores_compact}.

\begin{algorithm}[H]
\caption{Tensor leverage-score sampling}
\label{alg:lev_scores_sampling}
\begin{algorithmic}[1]
\Input A tensor \(\mathcal V\in\mathbb R^{n\times r\times p}\) with \(\mathcal V^\top*\mathcal V=\mathcal I_r\)
\Output Indices $J = (j_1,j_2, \ldots, j_r)$ for column subset selection;
    \State Initialize $J = ()$;
    \State Set \(p_j=\|\mathcal V(j,:,:)\|_F^2\) for \(j=1,\ldots,n\).
    \State Normalize \(p_j=p_j/\sum_{i=1}^{n}p_i\).
    \State Sample \(r\) distinct indices from \(\{1,\ldots,n\}\) with probabilities \(p_j\).
    \State \Return $J$
\end{algorithmic}
\end{algorithm}

\begin{algorithm}[H]
\caption{Tensor cross approximation with leverage-score sampling}
\label{alg:tcross_lev_scores_compact}
\begin{algorithmic}[1]
\Input A tensor \(\mathcal X\in\mathbb R^{n_1\times n_2\times n_3}\), a target rank \(r\leq\min(n_1,n_2)\), and a right basis \(\mathcal V\in\mathbb R^{n_2\times r\times n_3}\) satisfying \(\mathcal V^\top*\mathcal V=\mathcal I_r\).
\Output A reconstructed tensor \(\widehat{\mathcal X}\) of the same size as \(\mathcal X\).
\State \(J\leftarrow\text{TensorLeverageScoreSampling}(\mathcal V,r)\).
\State \(\mathcal C=\mathcal A_J\leftarrow\mathcal X(:,J,:)\).
\State \([\mathcal Q_J,\sim]\leftarrow\text{T-QR}(\mathcal A_J)\).
\State \(I\leftarrow\text{TensorLeverageScoreSampling}(\mathcal Q_J,r)\).
\State \(\mathcal R=\mathcal A_I\leftarrow\mathcal X(I,:,:)\).
\State \(\mathcal U\leftarrow\mathcal C^\dagger*\mathcal X*\mathcal R^\dagger\).
\State \(\widehat{\mathcal X}\leftarrow\mathcal C*\mathcal U*\mathcal R\).
\State \Return \(\widehat{\mathcal X}\).
\end{algorithmic}
\end{algorithm}

The three slice-sampling baselines use the reconstruction
\[
\widehat{\mathcal X}=\mathcal C*\mathcal C^\dagger*\mathcal X*\mathcal R^\dagger*\mathcal R.
\]
The truncated T-SVD is included separately as the best tubal-rank-\(r\) approximation in the Frobenius norm. It is not constrained to use actual lateral or horizontal slices.


\subsection{Results on Kodak dataset}
\label{subsec:results_on_kodak_dataset}

The Kodak dataset contains 24 RGB images, each in either landscape orientation, \(768\times512\times3\), or portrait orientation, \(512\times768\times3\). Pixel values are scaled from \([0,255]\) to \([0,1]\) before processing. The dataset is available at \url{https://r0k.us/graphics/kodak/}.

We compare ARP-T-CUR, T-ARP, and the Householder implementation of T-ARP with truncated T-SVD, uniform slice sampling, length-squared slice sampling, and leverage-score slice sampling. The baselines are defined in \Cref{subsec:decomposition_baselines}.

The algorithms are compared using \emph{relative error}, \emph{PSNR} and \emph{SSIM}.
The corresponding results for the methods listed in the tables are reported in
Table~\ref{tab:kodak_relative_error}, Table~\ref{tab:kodak_psnr}, and
Table~\ref{tab:kodak_ssim}, respectively; Figure~\ref{fig:kodak_metrics} also includes
length-squared sampling. These metrics
are defined in \Cref{subsec:exp_metrics}. Their mean values are also displayed in
Figure~\ref{fig:kodak_metrics}, and examples of image reconstructions are shown in
Figure~\ref{fig:img_reconstruction_examples}. The results show that both proposed constructions perform well against the considered
sampling baselines. ARP-T-CUR gives the lowest relative errors and the highest PSNR among
the sampling-based methods at all tested ranks. T-ARP and its Householder implementation
are particularly competitive for moderate and larger ranks. Their SSIM values are comparable
at small ranks and become better for moderate and larger ranks. As expected, the truncated T-SVD remains the best method in terms of
approximation error, since it is not constrained to use actual tensor slices. Figure~\ref{fig:video_reconstruction_examples}
shows an additional example of video reconstruction. This video experiment uses the higher-order product of Martin et al.~\cite{Martin2013}. It is an implementation extension of the third-order method; the theorems in Section~\ref{sec:tsvd} are stated only for the third-order t-product. The input is a fourth-order tensor \(\mathcal M\in\mathbb R^{H\times W\times C\times T}\), where \(H,W,C,T\) denote height, width, number of channels, and number of frames. In this example, the T-ARP reconstruction preserves the player and several faces in the background. The YUV dataset is accessible at \url{https://media.xiph.org/video/derf/}.

The next subsection reports the synthetic experiments.


\begin{table}[H]
\centering
\caption{Statistics of \emph{relative error} measures on Kodak dataset ($\text{mean} \pm \text{std}$)}
\label{tab:kodak_relative_error}
\footnotesize
\begin{tabular}{lcccccc}
\toprule
Target rank \(r\) & ARP-T-CUR & T-SVD & \makecell{T Uniform\\ Sampling} & \makecell{T Leverage\\ Scores Sampling} & T-ARP & \makecell{T-ARP\\(Householder)} \\
\midrule
10 & 0.257 $\pm$ 0.084 & 0.167 $\pm$ 0.065 & 0.274 $\pm$ 0.081 & 0.267 $\pm$ 0.081 & 0.269 $\pm$ 0.083 & 0.269 $\pm$ 0.083 \\
20 & 0.211 $\pm$ 0.079 & 0.135 $\pm$ 0.057 & 0.227 $\pm$ 0.079 & 0.225 $\pm$ 0.076 & 0.215 $\pm$ 0.081 & 0.215 $\pm$ 0.081 \\
40 & 0.170 $\pm$ 0.070 & 0.105 $\pm$ 0.048 & 0.185 $\pm$ 0.070 & 0.186 $\pm$ 0.068 & 0.178 $\pm$ 0.071 & 0.178 $\pm$ 0.071 \\
60 & 0.149 $\pm$ 0.065 & 0.087 $\pm$ 0.041 & 0.164 $\pm$ 0.064 & 0.163 $\pm$ 0.063 & 0.151 $\pm$ 0.064 & 0.151 $\pm$ 0.064 \\
80 & 0.130 $\pm$ 0.059 & 0.074 $\pm$ 0.037 & 0.146 $\pm$ 0.060 & 0.145 $\pm$ 0.057 & 0.134 $\pm$ 0.059 & 0.134 $\pm$ 0.059 \\
100 & 0.118 $\pm$ 0.055 & 0.064 $\pm$ 0.032 & 0.132 $\pm$ 0.056 & 0.131 $\pm$ 0.053 & 0.120 $\pm$ 0.054 & 0.120 $\pm$ 0.054 \\
\bottomrule
\end{tabular}
\end{table}

\begin{table}[H]
\centering
\caption{Statistics of \emph{PSNR} measures on Kodak dataset ($\text{mean} \pm \text{std}$)}
\label{tab:kodak_psnr}
\footnotesize
\begin{tabular}{lcccccc}
\toprule
Target rank \(r\) & ARP-T-CUR & T-SVD & \makecell{T Uniform\\ Sampling} & \makecell{T Leverage\\ Scores Sampling} & T-ARP & \makecell{T-ARP\\(Householder)} \\
\midrule
10 & 19.018 $\pm$ 2.243 & 22.907 $\pm$ 2.585 & 18.366 $\pm$ 2.057 & 18.710 $\pm$ 2.281 & 18.568 $\pm$ 2.090 & 18.568 $\pm$ 2.090 \\
20 & 20.842 $\pm$ 2.486 & 24.825 $\pm$ 2.773 & 20.118 $\pm$ 2.303 & 20.202 $\pm$ 2.291 & 20.680 $\pm$ 2.563 & 20.680 $\pm$ 2.563 \\
40 & 22.779 $\pm$ 2.704 & 27.142 $\pm$ 3.013 & 21.981 $\pm$ 2.462 & 22.080 $\pm$ 2.517 & 22.380 $\pm$ 2.701 & 22.380 $\pm$ 2.701 \\
60 & 24.064 $\pm$ 2.886 & 28.830 $\pm$ 3.203 & 23.063 $\pm$ 2.542 & 23.432 $\pm$ 2.692 & 23.895 $\pm$ 2.884 & 23.895 $\pm$ 2.884 \\
80 & 25.267 $\pm$ 3.074 & 30.278 $\pm$ 3.360 & 24.120 $\pm$ 2.671 & 24.521 $\pm$ 2.789 & 24.989 $\pm$ 2.949 & 24.989 $\pm$ 2.949 \\
100 & 26.190 $\pm$ 3.139 & 31.605 $\pm$ 3.484 & 25.045 $\pm$ 2.782 & 25.506 $\pm$ 2.918 & 25.932 $\pm$ 3.025 & 25.932 $\pm$ 3.025 \\
\bottomrule
\end{tabular}
\end{table}

\begin{table}[H]
\centering
\caption{Statistics of \emph{SSIM} measures on Kodak dataset ($\text{mean} \pm \text{std}$)}
\label{tab:kodak_ssim}
\footnotesize
\begin{tabular}{lcccccc}
\toprule
Target rank \(r\) & ARP-T-CUR & T-SVD & \makecell{T Uniform\\ Sampling} & \makecell{T Leverage\\ Scores Sampling} & T-ARP & \makecell{T-ARP\\(Householder)} \\
\midrule
10 & 0.428 $\pm$ 0.130 & 0.589 $\pm$ 0.128 & 0.426 $\pm$ 0.135 & 0.431 $\pm$ 0.134 & 0.423 $\pm$ 0.130 & 0.423 $\pm$ 0.130 \\
20 & 0.471 $\pm$ 0.130 & 0.643 $\pm$ 0.113 & 0.465 $\pm$ 0.128 & 0.466 $\pm$ 0.126 & 0.471 $\pm$ 0.132 & 0.471 $\pm$ 0.132 \\
40 & 0.536 $\pm$ 0.123 & 0.718 $\pm$ 0.094 & 0.522 $\pm$ 0.120 & 0.519 $\pm$ 0.121 & 0.523 $\pm$ 0.124 & 0.523 $\pm$ 0.124 \\
60 & 0.586 $\pm$ 0.117 & 0.771 $\pm$ 0.080 & 0.563 $\pm$ 0.113 & 0.564 $\pm$ 0.114 & 0.583 $\pm$ 0.118 & 0.583 $\pm$ 0.118 \\
80 & 0.631 $\pm$ 0.110 & 0.812 $\pm$ 0.070 & 0.605 $\pm$ 0.106 & 0.603 $\pm$ 0.104 & 0.628 $\pm$ 0.110 & 0.628 $\pm$ 0.110 \\
100 & 0.670 $\pm$ 0.101 & 0.844 $\pm$ 0.060 & 0.643 $\pm$ 0.099 & 0.641 $\pm$ 0.096 & 0.668 $\pm$ 0.102 & 0.668 $\pm$ 0.102 \\
\bottomrule
\end{tabular}
\end{table}

\begin{figure}[ht!]
\centering
\begin{subfigure}{0.9\textwidth}
    \centering
    \resizebox{!}{0.25\textheight}
    {
\begin{tikzpicture}

\definecolor{crimson2143940}{RGB}{214,39,40}
\definecolor{darkgray176}{RGB}{176,176,176}
\definecolor{darkorange25512714}{RGB}{255,127,14}
\definecolor{forestgreen4416044}{RGB}{44,160,44}
\definecolor{lightgray204}{RGB}{204,204,204}
\definecolor{mediumpurple148103189}{RGB}{148,103,189}
\definecolor{orchid227119194}{RGB}{227,119,194}
\definecolor{sienna1408675}{RGB}{140,86,75}
\definecolor{steelblue31119180}{RGB}{31,119,180}

\begin{axis}[
width=\textwidth, height=0.625\textwidth,
legend cell align={left},
legend style={
  fill opacity=0.8,
  draw opacity=1,
  text opacity=1,
  at={(0.03,0.03)},
  anchor=south west,
  draw=lightgray204,
  font=\small
},
log basis y={10},
minor ytick={0.002,0.003,0.004,0.005,0.006,0.007,0.008,0.009,0.02,0.03,0.04,0.05,0.06,0.07,0.08,0.09,0.2,0.3,0.4,0.5,0.6,0.7,0.8,0.9,2,3,4,5,6,7,8,9,20,30,40,50,60,70,80,90},
tick align=outside,
tick pos=left,
title={Reconstruction Relative Error – Lower is better},
x grid style={darkgray176},
xlabel={Target rank $r$},
xmajorgrids,
xmin=5.5, xmax=104.5,
xtick style={color=black},
y grid style={darkgray176},
ylabel={Relative Error (log scale)},
ymajorgrids,
ymin=0.0591915555001068, ymax=0.38777519120125,
ymode=log,
ytick style={color=black}
]
\addplot [thick, steelblue31119180, mark=o, mark size=3, mark options={solid}]
table {%
10 0.256731335709478
20 0.210837695014597
40 0.170470984850853
60 0.148526801070681
80 0.13033801322695
100 0.11771161003889
};
\addlegendentry{ARP-T-CUR}
\addplot [thick, darkorange25512714, mark=*, mark size=3, mark options={solid}]
table {%
10 0.166817981324622
20 0.13520921933312
40 0.104889181672542
60 0.0871641512130016
80 0.0743297286064655
100 0.0641776850612685
};
\addlegendentry{T-SVD}
\addplot [thick, forestgreen4416044, mark=triangle, mark size=3, mark options={solid}]
table {%
10 0.273577474766627
20 0.226778700451988
40 0.184733157563646
60 0.163867531293315
80 0.145905576736699
100 0.131799650337248
};
\addlegendentry{T Uniform Sampling}
\addplot [thick, orchid227119194, mark=square, mark size=3, mark options={solid}]
table {%
10 0.267166414538512
20 0.225017744635527
40 0.186457808500992
60 0.162798409104007
80 0.145370552544989
100 0.130921141845928
};
\addlegendentry{T Lengths Squared Sampling}
\addplot [thick, crimson2143940, mark=diamond, mark size=3, mark options={solid}]
table {%
10 0.265597091169927
20 0.224620569015835
40 0.183139346066311
60 0.158192901500172
80 0.140357993878808
100 0.126043519417637
};
\addlegendentry{T Leverage Scores Sampling}
\addplot [thick, mediumpurple148103189, mark=+, mark size=6, mark options={solid}]
table {%
10 0.268508827613968
20 0.215071564028095
40 0.178045676961747
60 0.150914811120027
80 0.133701742626005
100 0.120499267943919
};
\addlegendentry{T-ARP}
\addplot [thick, sienna1408675, mark=x, mark size=6, mark options={solid}]
table {%
10 0.268508827613968
20 0.215071564028095
40 0.178045676961747
60 0.150914811120027
80 0.133701742626005
100 0.120499267943919
};
\addlegendentry{T-ARP (Householder)}

\end{axis}

\end{tikzpicture}}
    \caption{Relative error}
\end{subfigure}
\\[1ex]   
\begin{subfigure}{0.9\textwidth}
    \centering
    \resizebox{!}{0.25\textheight}
    {
\begin{tikzpicture}

\definecolor{crimson2143940}{RGB}{214,39,40}
\definecolor{darkgray176}{RGB}{176,176,176}
\definecolor{darkorange25512714}{RGB}{255,127,14}
\definecolor{forestgreen4416044}{RGB}{44,160,44}
\definecolor{lightgray204}{RGB}{204,204,204}
\definecolor{mediumpurple148103189}{RGB}{148,103,189}
\definecolor{orchid227119194}{RGB}{227,119,194}
\definecolor{sienna1408675}{RGB}{140,86,75}
\definecolor{steelblue31119180}{RGB}{31,119,180}

\begin{axis}[
width=\textwidth, height=0.625\textwidth,
legend cell align={left},
legend style={
  fill opacity=0.8,
  draw opacity=1,
  text opacity=1,
  at={(0.03,0.97)},
  anchor=north west,
  draw=lightgray204,
  font=\small
},
tick align=outside,
tick pos=left,
title={Peak Signal‑to‑Noise Ratio (PSNR) – Higher is better},
x grid style={darkgray176},
xlabel={Target rank $r$},
xmajorgrids,
xmin=5.5, xmax=104.5,
xtick style={color=black},
y grid style={darkgray176},
ylabel={PSNR (dB)},
ymajorgrids,
ymin=16.0424859325091, ymax=32.3465323646863,
ytick style={color=black}
]
\addplot [thick, steelblue31119180, mark=o, mark size=2, mark options={solid}]
table {%
10 19.0180056970953
20 20.8421286614876
40 22.7791543162373
60 24.0637875781074
80 25.2665552443878
100 26.1897226124334
};
\addlegendentry{ARP-T-CUR}
\addplot [thick, darkorange25512714, mark=*, mark size=3, mark options={solid}]
table {%
10 22.9067511597666
20 24.8253412999205
40 27.1423579307737
60 28.8302777884521
80 30.2778173827698
100 31.6054374122803
};
\addlegendentry{T-SVD}
\addplot [thick, forestgreen4416044, mark=triangle, mark size=3, mark options={solid}]
table {%
10 18.3662416322001
20 20.1182959636792
40 21.9808643109481
60 23.063479119885
80 24.1197132014432
100 25.0450812803085
};
\addlegendentry{T Uniform Sampling}
\addplot [thick, orchid227119194, mark=square, mark size=3, mark options={solid}]
table {%
10 18.5976187399001
20 20.1502821253142
40 21.8669568505643
60 23.1045151598978
80 24.1080839799737
100 25.0452131366168
};
\addlegendentry{T Lengths Squared Sampling}
\addplot [thick, crimson2143940, mark=diamond, mark size=3, mark options={solid}]
table {%
10 18.7095300192077
20 20.2018900072356
40 22.0799334054492
60 23.4321751207842
80 24.5213990622971
100 25.5060771772351
};
\addlegendentry{T Leverage Scores Sampling}
\addplot [thick, mediumpurple148103189, mark=+, mark size=6, mark options={solid}]
table {%
10 18.5675172983283
20 20.6799291242019
40 22.3802049555358
60 23.8951743011135
80 24.9886514605087
100 25.9316678347181
};
\addlegendentry{T-ARP}
\addplot [thick, sienna1408675, mark=x, mark size=6, mark options={solid}]
table {%
10 18.5675172983283
20 20.6799291242019
40 22.3802049555358
60 23.8951743011135
80 24.9886514605087
100 25.9316678347181
};
\addlegendentry{T-ARP (Householder)}
\end{axis}

\end{tikzpicture}}
    \caption{PSNR}
\end{subfigure}
\\[1ex]
\begin{subfigure}{0.9\textwidth}
    \centering
    \resizebox{!}{0.25\textheight}
    {
\begin{tikzpicture}

\definecolor{crimson2143940}{RGB}{214,39,40}
\definecolor{darkgray176}{RGB}{176,176,176}
\definecolor{darkorange25512714}{RGB}{255,127,14}
\definecolor{forestgreen4416044}{RGB}{44,160,44}
\definecolor{lightgray204}{RGB}{204,204,204}
\definecolor{mediumpurple148103189}{RGB}{148,103,189}
\definecolor{orchid227119194}{RGB}{227,119,194}
\definecolor{sienna1408675}{RGB}{140,86,75}
\definecolor{steelblue31119180}{RGB}{31,119,180}

\begin{axis}[
width=\textwidth, height=0.625\textwidth,
legend cell align={left},
legend style={
  fill opacity=0.8,
  draw opacity=1,
  text opacity=1,
  at={(0.03,0.97)},
  anchor=north west,
  draw=lightgray204,
  font=\small
},
tick align=outside,
tick pos=left,
title={Structural Similarity Index (SSIM) – Higher is better},
x grid style={darkgray176},
xlabel={Target rank $r$},
xmajorgrids,
xmin=5.5, xmax=104.5,
xtick style={color=black},
y grid style={darkgray176},
ylabel={SSIM},
ymajorgrids,
ymin=0.364917129402359, ymax=0.867174548283219,
ytick style={color=black}
]
\addplot [thick, steelblue31119180, mark=o, mark size=3, mark options={solid}]
table {%
10 0.428145568891612
20 0.471409867458581
40 0.536296147172191
60 0.586131053553626
80 0.631230108144373
100 0.669844994844149
};
\addlegendentry{ARP-T-CUR}
\addplot [thick, darkorange25512714, mark=*, mark size=3, mark options={solid}]
table {%
10 0.588897641689434
20 0.642763798709749
40 0.717836630692143
60 0.770844666080352
80 0.811562490985377
100 0.844346078234332
};
\addlegendentry{T-SVD}
\addplot [thick, forestgreen4416044, mark=triangle, mark size=3, mark options={solid}]
table {%
10 0.426453906634748
20 0.465051760627376
40 0.522308601994825
60 0.56341784803574
80 0.605157303776503
100 0.643223965774472
};
\addlegendentry{T Uniform Sampling}
\addplot [thick, orchid227119194, mark=square, mark size=3, mark options={solid}]
table {%
10 0.430615033876279
20 0.465929472869091
40 0.519436344803934
60 0.563873114349012
80 0.602661917515583
100 0.640900442553049
};
\addlegendentry{T Lengths Squared Sampling}
\addplot [thick, crimson2143940, mark=diamond, mark size=3, mark options={solid}]
table {%
10 0.421721864223572
20 0.454611695081742
40 0.512979178290881
60 0.563898120815407
80 0.607732457642798
100 0.64921730226709
};
\addlegendentry{T Leverage Scores Sampling}
\addplot [thick, mediumpurple148103189, mark=+, mark size=6, mark options={solid}]
table {%
10 0.422797222357265
20 0.470552692546765
40 0.523177266902036
60 0.583182683879592
80 0.628474296659119
100 0.667525874879744
};
\addlegendentry{T-ARP}
\addplot [thick, sienna1408675, mark=x, mark size=6, mark options={solid}]
table {%
10 0.422797222357265
20 0.470552692546765
40 0.523177266902036
60 0.583182683879592
80 0.628474296659119
100 0.667525874879744
};
\addlegendentry{T-ARP (Householder)}
\end{axis}

\end{tikzpicture}}
    \caption{SSIM}
\end{subfigure}
\caption{Visualization of mean values of metrics on the Kodak dataset. \emph{Relative error}, \emph{PSNR}, and \emph{SSIM} are shown in subfigures (a), (b), and (c), respectively. The horizontal axis reports the target rank \(r\); for the common-index methods, \(r\) is also the number of selected lateral and horizontal slices, while ARP-T-CUR uses \(r\) row and column indices at each nonredundant Fourier frequency.}\label{fig:kodak_metrics}
\end{figure}

\subsection{Results on Synthetic Data}
\label{subsec:results_on_synthetic_data}

The synthetic experiments compare ARP-T-CUR and the two implementations of T-ARP with truncated T-SVD, uniform slice sampling, and leverage-score slice sampling. T-SVD is included as an accuracy reference and is not constrained to use actual tensor slices. To isolate the effect of slice selection, the basis-based methods use bases computed from the exact T-SVD of each noisy tensor. All sampling-based methods, including ARP-T-CUR, are evaluated with the two-sided projection used in the practical benchmarks. These curves therefore compare the selected subspaces; they do not directly test the inverse-based ARP-T-CUR bound of Theorem~\ref{thm:arp-tsvd-error}. For every target rank, we report the mean relative error over 20 independent realizations. Only the mean curves are displayed.

In both experiments, the noiseless tensor \(\mathcal Z_0\) is perturbed by Gaussian noise with a prescribed relative Frobenius norm. More precisely, we set
\[
\mathcal Z
=
\mathcal Z_0
+
\varepsilon
\frac{\|\mathcal Z_0\|_F}{\|\mathcal E\|_F}
\mathcal E,
\qquad
\varepsilon=10^{-6},
\]
where the entries of $\mathcal E$ are independent standard Gaussian random variables. Hence,
\[
\frac{\|\mathcal Z-\mathcal Z_0\|_F}{\|\mathcal Z_0\|_F}
=
10^{-6}.
\]

\paragraph{Function-based tensor}
We first consider the noiseless third-order tensor \(\mathcal X_0\in\mathbb R^{m\times n\times p}\) defined by
\[
\mathcal X_0(i,j,\ell)
=
\frac{1}{(i+j+\ell)^{1/\alpha}},
\qquad \alpha\geq 1.
\]
These tensors are numerically low tubal rank: their tubal singular values decay rapidly, although their exact tubal rank depends on the numerical tolerance. We use $\mathcal X_0\in\mathbb R^{60\times60\times60}$ with $\alpha=2$, and generate 20 noisy tensors by drawing 20 independent perturbations $\mathcal E$.

The results are shown in Figure~\ref{fig:results_on_synthetic_data}(a). As expected, T-SVD gives the smallest error and reaches the noise level at a small target rank. Among the methods based on actual slices, T-ARP generally gives the lowest error for moderate and larger target ranks. Its algebraic and Householder implementations produce indistinguishable curves. ARP-T-CUR is also competitive: it performs similarly to the other adaptive sampling methods and improves substantially over uniform sampling for most tested ranks.

\paragraph{Random tensor}
We next generate a noiseless random tensor from the factor model
\[
\mathcal Y_0
=
\mathcal A*\mathcal B,
\]
where \(\mathcal A\in\mathbb R^{m\times r_0\times p}\) and \(\mathcal B\in\mathbb R^{r_0\times n\times p}\) have independent entries drawn uniformly from \([0,1]\). The tensor \(\mathcal Y_0\) has tubal rank at most \(r_0\). In the reported realizations, its numerical tubal rank is \(r_0=7\). We use \(m=n=p=60\). For each of the 20 trials, the factors \(\mathcal A\) and \(\mathcal B\), as well as the relative Gaussian perturbation, are generated independently.

Figure~\ref{fig:results_on_synthetic_data}(b) shows the results. The T-SVD error reaches the noise level when the target rank reaches \(r_0=7\). The slice-selection methods require a few additional slices, but all of them reach errors of order \(10^{-6}\) when approximately 10--12 lateral and horizontal slices are selected. ARP-T-CUR is comparable to the common-index methods in this experiment. The algebraic and Householder T-ARP curves again coincide, which illustrates their numerical equivalence in these tests.

\begin{figure}[ht!]
\centering
\begin{subfigure}{0.48\textwidth}
    \centering
    \resizebox{\linewidth}{!}{
\begin{tikzpicture}

\definecolor{crimson2143940}{RGB}{214,39,40}
\definecolor{darkgray176}{RGB}{176,176,176}
\definecolor{darkorange25512714}{RGB}{255,127,14}
\definecolor{forestgreen4416044}{RGB}{44,160,44}
\definecolor{lightgray204}{RGB}{204,204,204}
\definecolor{mediumpurple148103189}{RGB}{148,103,189}
\definecolor{sienna1408675}{RGB}{140,86,75}
\definecolor{steelblue31119180}{RGB}{31,119,180}

\begin{axis}[
legend cell align={left},
legend style={fill opacity=0.8, draw opacity=1, text opacity=1, draw=lightgray204, font=\tiny
},
log basis y={10},
tick align=outside,
tick pos=left,
title={Reconstruction Relative Error – Lower is better},
x grid style={darkgray176},
xlabel={\# Slices},
xmajorgrids,
xmin=1, xmax=23,
xtick style={color=black},
y grid style={darkgray176},
ylabel={Relative Error (log scale)},
ymajorgrids,
ymin=1.35912616564645e-07, ymax=0.041668268042864,
ymode=log,
ytick style={color=black}
]
\addplot [thick, steelblue31119180, mark=o, mark size=2, mark options={solid}]
table {%
2 0.00684985023705262
3 0.00103956086576085
4 0.000122300827113568
5 1.89405424275876e-05
6 3.29212079175371e-06
7 2.18239026742877e-06
8 1.85354186818044e-06
9 1.67244680419419e-06
10 1.50064227328843e-06
11 1.42333950785098e-06
12 1.38197888082361e-06
13 1.30386008451256e-06
14 1.21560450611012e-06
15 1.18923990301018e-06
16 1.13164042992546e-06
17 1.10719039023551e-06
18 1.06478890808948e-06
19 1.02180064140061e-06
20 1.00304669563682e-06
21 9.69178116583512e-07
22 9.32386448407346e-07
};
\addlegendentry{ARP T-CUR}
\addplot [thick, darkorange25512714, mark=*, mark size=2, mark options={solid}]
table {%
2 0.00321470293584038
3 0.000365992913455423
4 4.43233850716192e-05
5 5.41486494663476e-06
6 1.09461304737164e-06
7 8.84959293943038e-07
8 8.53062679025467e-07
9 8.23376842816498e-07
10 7.94898907200398e-07
11 7.67366584767411e-07
12 7.40596490593475e-07
13 7.14550557457734e-07
14 6.89188839911323e-07
15 6.64409472228148e-07
16 6.4019417927914e-07
17 6.16538599915413e-07
18 5.93372773942143e-07
19 5.70681173567675e-07
20 5.48453388535584e-07
21 5.26690421312239e-07
22 5.05332602828905e-07
};
\addlegendentry{T-SVD}
\addplot [thick, forestgreen4416044, mark=triangle, mark size=2, mark options={solid}]
table {%
2 0.0106373571556392
3 0.00295393948073906
4 0.00141471114504976
5 0.000859723895735241
6 0.000571000354050475
7 0.000414058399822025
8 0.000294000058557287
9 0.000246828391601505
10 0.000194344167381091
11 0.000179814220419711
12 0.000165911931668119
13 0.000117727444857063
14 6.75858354744305e-05
15 5.42431183088924e-05
16 4.84563379548514e-05
17 4.64052407568414e-05
18 4.18776043011734e-05
19 4.02072159240253e-05
20 3.84608768761581e-05
21 2.97812122414507e-05
22 2.16588600156883e-05
};
\addlegendentry{T Uniform Sampling}
\addplot [thick, crimson2143940, mark=diamond, mark size=2, mark options={solid}]
table {%
2 0.0110036889984432
3 0.00200220264007299
4 0.00031934068583695
5 0.000120460688208996
6 3.67982945624459e-05
7 1.88682962125605e-05
8 2.06909883736607e-05
9 1.8379597934418e-05
10 7.85359543136375e-06
11 1.01702327613978e-05
12 7.74925199301469e-06
13 1.13994958569423e-05
14 1.0951954249872e-05
15 1.06201695941066e-05
16 5.78605546132703e-06
17 6.08444306374053e-06
18 4.84804073816709e-06
19 4.43355213320183e-06
20 4.159044980376e-06
21 4.00901363335226e-06
22 3.90128713300835e-06
};
\addlegendentry{T Leverage Scores Sampling}
\addplot [thick, mediumpurple148103189, mark=+, mark size=4, mark options={solid}]
table {%
2 0.00734069342012056
3 0.000949736277238376
4 0.000129940203618802
5 1.61549156577029e-05
6 3.81521317171035e-06
7 3.17934550642503e-06
8 1.84927744944129e-06
9 1.91805691818239e-06
10 1.58741491784274e-06
11 1.34732256301275e-06
12 1.45392120074381e-06
13 1.24514828047432e-06
14 1.26825125817742e-06
15 1.22709948524393e-06
16 1.17229482798663e-06
17 1.14451719425807e-06
18 1.10620848218794e-06
19 1.26101771217806e-06
20 1.00175796572477e-06
21 9.62142563474364e-07
22 9.83372700085864e-07
};
\addlegendentry{T-ARP}
\addplot [thick, sienna1408675, mark=x, mark size=4, mark options={solid}]
table {%
2 0.00734069342012056
3 0.000949736277238376
4 0.000129940203618802
5 1.61549156577029e-05
6 3.81521317171035e-06
7 3.17934550642503e-06
8 1.84927744944129e-06
9 1.91805691818239e-06
10 1.58741491784274e-06
11 1.34732256301275e-06
12 1.45392120074381e-06
13 1.24514828047432e-06
14 1.26825125817742e-06
15 1.22709948524393e-06
16 1.17229482798663e-06
17 1.14451719425807e-06
18 1.10620848218794e-06
19 1.26101771217806e-06
20 1.00175796572477e-06
21 9.62142563474364e-07
22 9.83372700085864e-07
};
\addlegendentry{T-ARP (Householder)}
\end{axis}

\end{tikzpicture}}
    \caption{Function-based tensor.}
\end{subfigure}
\hfill
\begin{subfigure}{0.48\textwidth}
    \centering
    \resizebox{\linewidth}{!}{
\begin{tikzpicture}

\definecolor{crimson2143940}{RGB}{214,39,40}
\definecolor{darkgray176}{RGB}{176,176,176}
\definecolor{darkorange25512714}{RGB}{255,127,14}
\definecolor{forestgreen4416044}{RGB}{44,160,44}
\definecolor{lightgray204}{RGB}{204,204,204}
\definecolor{mediumpurple148103189}{RGB}{148,103,189}
\definecolor{sienna1408675}{RGB}{140,86,75}
\definecolor{steelblue31119180}{RGB}{31,119,180}

\begin{axis}[
legend cell align={left},
legend style={fill opacity=0.8, draw opacity=1, text opacity=1, draw=lightgray204, font=\tiny
},
log basis y={10},
tick align=outside,
tick pos=left,
title={Reconstruction Relative Error – Lower is better},
x grid style={darkgray176},
xlabel={\# Slices},
xmajorgrids,
xmin=1, xmax=23,
xtick style={color=black},
y grid style={darkgray176},
ylabel={Relative Error (log scale)},
ymajorgrids,
ymin=2.60965294921384e-07, ymax=0.0261349762978426,
ymode=log,
ytick style={color=black}
]
\addplot [thick, steelblue31119180, mark=o, mark size=2, mark options={solid}]
table {%
2 0.0149518621964274
3 0.0136245257788389
4 0.011956123984121
5 0.00987724502604817
6 0.007025451453657
7 3.44376527306146e-06
8 2.42767779855521e-06
9 2.03478600056499e-06
10 1.78803399490739e-06
11 1.61679470405241e-06
12 1.48851801860183e-06
13 1.39632290250975e-06
14 1.32392222758019e-06
15 1.25366505326572e-06
16 1.19521469276614e-06
17 1.14657420428862e-06
18 1.10457369376886e-06
19 1.06355563690568e-06
20 1.02458861596128e-06
21 9.91497522540367e-07
22 9.57420518506575e-07
};
\addlegendentry{ARP T-CUR}
\addplot [thick, darkorange25512714, mark=*, mark size=2, mark options={solid}]
table {%
2 0.0118851367527048
3 0.00991760338600087
4 0.00799986590475393
5 0.00604863941757958
6 0.00392514491569863
7 8.83472965591847e-07
8 8.52213257530488e-07
9 8.226472980225e-07
10 7.94191162477645e-07
11 7.66657838974267e-07
12 7.39968390166858e-07
13 7.13999448945432e-07
14 6.88649938557673e-07
15 6.63902783559852e-07
16 6.39764835627821e-07
17 6.16128132222979e-07
18 5.9297840168743e-07
19 5.70340779634772e-07
20 5.48154486486989e-07
21 5.26392574515138e-07
22 5.05059960139362e-07
};
\addlegendentry{T-SVD}
\addplot [thick, forestgreen4416044, mark=triangle, mark size=2, mark options={solid}]
table {%
2 0.0153523904967098
3 0.0141369000394757
4 0.0125489615781654
5 0.0104869216187548
6 0.00755503714895819
7 1.22618434699412e-05
8 3.57647675547134e-06
9 2.44787695107769e-06
10 2.02728775026779e-06
11 1.77299952221354e-06
12 1.60605313275746e-06
13 1.47880069017317e-06
14 1.38201642329869e-06
15 1.30400308549622e-06
16 1.23931812757033e-06
17 1.18374217914981e-06
18 1.13417652795429e-06
19 1.0899518335699e-06
20 1.04876213924343e-06
21 1.01113088193417e-06
22 9.76847824819841e-07
};
\addlegendentry{T Uniform Sampling}
\addplot [thick, crimson2143940, mark=diamond, mark size=2, mark options={solid}]
table {%
2 0.015311897806069
3 0.0141364794422601
4 0.0125468554226836
5 0.0104672883022708
6 0.00754345532740826
7 1.25660352792772e-05
8 3.59775908015059e-06
9 2.46165909511436e-06
10 2.02472033993272e-06
11 1.77381275543029e-06
12 1.6059877849457e-06
13 1.47895202013515e-06
14 1.38018749496769e-06
15 1.30463558660489e-06
16 1.23913222276629e-06
17 1.18347428303187e-06
18 1.13384031135125e-06
19 1.08979091071175e-06
20 1.04890984954308e-06
21 1.01108852907119e-06
22 9.76585157444519e-07
};
\addlegendentry{T Leverage Scores Sampling}
\addplot [thick, mediumpurple148103189, mark=+, mark size=4, mark options={solid}]
table {%
2 0.0153387178008227
3 0.0141538715511427
4 0.0125547572797964
5 0.0104929429994219
6 0.00753153305590553
7 1.40099764606389e-05
8 3.42221390471236e-06
9 2.41769141576789e-06
10 1.99852017816572e-06
11 1.7603338041131e-06
12 1.59762819070631e-06
13 1.47135696080771e-06
14 1.38095810361669e-06
15 1.30002329129334e-06
16 1.23745047938511e-06
17 1.18066143893566e-06
18 1.13236265958715e-06
19 1.08749672551026e-06
20 1.04720134452903e-06
21 1.00967820916788e-06
22 9.74660341436771e-07
};
\addlegendentry{T-ARP}
\addplot [thick, sienna1408675, mark=x, mark size=4, mark options={solid}]
table {%
2 0.0153387178008227
3 0.0141538715511427
4 0.0125547572797964
5 0.0104929429994219
6 0.00753153305590553
7 1.40099764606389e-05
8 3.42221390471236e-06
9 2.41769141576789e-06
10 1.99852017816572e-06
11 1.7603338041131e-06
12 1.59762819070631e-06
13 1.47135696080771e-06
14 1.38095810361669e-06
15 1.30002329129334e-06
16 1.23745047938511e-06
17 1.18066143893566e-06
18 1.13236265958715e-06
19 1.08749672551026e-06
20 1.04720134452903e-06
21 1.00967820916788e-06
22 9.74660341436771e-07
};
\addlegendentry{T-ARP (Householder)}
\end{axis}

\end{tikzpicture}}
    \caption{Random factor-model tensor.}
\end{subfigure}
\caption{Mean reconstruction relative errors over 20 independent realizations as a function of the target rank \(r\). For the common-index methods, \(r\) is also the number of lateral and horizontal slices selected; ARP-T-CUR uses \(r\) row and column indices at each nonredundant Fourier frequency. In both experiments, the noiseless tensor is perturbed by Gaussian noise with relative Frobenius norm \(10^{-6}\). Subfigure~(a) shows the results for a function-based tensor \(\mathcal X_0\in\mathbb R^{60\times60\times60}\) with \(\alpha=2\). Subfigure~(b) shows the results for a random tensor generated from a factor model with \(r_0=7\) and size \(60\times60\times60\).}
\label{fig:results_on_synthetic_data}
\end{figure}

\subsection{Numerical validation of the common-index bound}
\label{subsec:empirical-tarp-bound}

We complement the previous approximation experiments with a small test that uses
exactly the randomized approximation appearing in
Theorem~\ref{thm:tarp-error}. We consider
\[
\mathcal X\in\mathbb R^{10\times 8\times 3},
\qquad
\mathcal V\in\mathbb R^{8\times 3\times 3},
\qquad
r=3.
\]
For every ordered index sequence \(J\) returned by randomized T-ARP, we compute
the normalized squared error
\[
R(J)
=
\frac{
\left\|
\mathcal X
-
\mathcal X(:,J,:)
*
\bigl(\mathcal V(J,:,:)\bigr)^{-\top}
*
\mathcal V^\top
\right\|_F^2
}{
\left\|
\mathcal X
-
\mathcal X
*
\mathcal V
*
\mathcal V^\top
\right\|_F^2
}.
\]
Theorem~\ref{thm:tarp-error} gives
\[
\mathbb E[R(J)]
\leq
\Gamma_\eta.
\]

The test tensors are constructed in the Fourier domain. For each frequency,
the columns of \(\widehat V^{(\ell)}\) are orthonormal, and we set
\[
\widehat X^{(\ell)}
=
G^{(\ell)}
\bigl(\widehat V^{(\ell)}\bigr)^H
+
0.15\,E^{(\ell)}
\left(
I
-
\widehat V^{(\ell)}
\bigl(\widehat V^{(\ell)}\bigr)^H
\right),
\]
where \(G^{(\ell)}\) and \(E^{(\ell)}\) are Gaussian matrices. Conjugate
symmetry is imposed so that \(\mathcal X\) and \(\mathcal V\) are real.

We consider two cases. In the aligned case, the same right basis is used at
every Fourier frequency. In the mildly misaligned case, we perturb the basis at one member of the non-DC conjugate pair by adding a complex Gaussian perturbation of magnitude \(0.02\), followed by reorthogonalization. The conjugate basis is used at the paired frequency.

Since \(n=8\) and \(r=3\), all
\[
8\cdot 7\cdot 6=336
\]
ordered sequences of three distinct indices can be enumerated. We use this
enumeration to verify the support, frequency-wise rank-growth, and final
nonsingularity assumptions for every possible history. It also allows us to
compute each \(\eta_k\) as the maximum required ratio over all positive-probability histories, frequencies, and admissible indices, together with the factor \(\Gamma_\eta\), and the 
full-enumeration expectation of \(R(J)\), which is exact for this finite algorithmic distribution up to floating-point arithmetic. As a Monte Carlo check, we draw \(100{,}000\) ordered histories from the fully enumerated T-ARP distribution and estimate the same expectation.

In the aligned case, the frequency-wise leverage-score distributions
coincide. We obtain
\[
\eta_1=\eta_2=\eta_3=1,
\qquad
\Gamma_\eta=4.
\]
The full-enumeration normalized expectation is
\[
\mathbb E[R(J)]=4,
\]
so the \(r+1\) bound is attained in this example. The Monte Carlo estimate is
\(4.044\), with standard error \(0.045\), and is therefore consistent with
the full-enumeration value.

In the mildly misaligned case, we obtain
\[
(\eta_1,\eta_2,\eta_3)
=
(1.028,1.138,2.795),
\qquad
\Gamma_\eta=7.994.
\]
The full-enumeration expectation is \(4.045\), while the Monte Carlo estimate is
\(4.066\), with standard error \(0.022\). The theoretical bound remains
valid, although it becomes more conservative when the frequency-wise
sampling distributions are not perfectly aligned.

\begin{figure}[htbp]
\centering
\resizebox{0.78\linewidth}{!}{%
\definecolor{empiricalBlue}{RGB}{0,114,178}
\definecolor{boundOrange}{RGB}{230,159,0}
\begin{tikzpicture}
\begin{axis}[
width=0.78\linewidth,
height=0.55\linewidth,
ybar=3pt,
bar width=15pt,
ylabel={Normalized squared error},
symbolic x coords={Aligned,Mildly misaligned},
xtick=data,
ymin=0,
grid=major,
legend style={font=\scriptsize, at={(0.5,-0.22)}, anchor=north, legend columns=3},
]
\addplot+[fill=empiricalBlue, draw=empiricalBlue] coordinates {(Aligned,4.043721627892e+00) (Mildly misaligned,4.066138458756e+00)};
\addlegendentry{Empirical mean}
\addplot+[fill=boundOrange, draw=boundOrange] coordinates {(Aligned,4.000000000000e+00) (Mildly misaligned,7.993872750631e+00)};
\addlegendentry{Theoretical bound}
\addplot+[only marks, mark=diamond*, black, mark size=2.8pt] coordinates {(Aligned,4.000000000000e+00) (Mildly misaligned,4.044559518512e+00)};
\addlegendentry{Exact expectation}
\end{axis}
\end{tikzpicture}%
}
\caption{Numerical validation of Theorem~\ref{thm:tarp-error}.
The bars compare the full-enumeration normalized expectations with the theoretical
factors \(\Gamma_\eta\). The points show Monte Carlo means obtained from
\(100{,}000\) samples drawn from the fully enumerated T-ARP distribution. The corresponding standard errors are reported in the text.}
\label{fig:empirical-tarp-bound}
\end{figure}

\section{Conclusion and future work}
\label{sec:conclusion}

We have developed two extensions of Adaptive Randomized Pivoting in the t-product framework. ARP-T-CUR applies matrix ARP-cross at the nonredundant Fourier frequencies and couples conjugate pairs to preserve real-valued tensors. T-ARP instead selects common lateral and horizontal slices, with the same indices used at every Fourier frequency. It therefore gives a tensor cross approximation with indices shared across all frequencies.

The main difficulty in the analysis of T-ARP comes from this common-index constraint. Under frequency-alignment and frequency-wise rank-growth assumptions, we prove an expected-error bound for one-sided T-ARP and derive corresponding bounds for two-stage tensor cross approximation and t-DEIM. When the frequency-wise sampling distributions are perfectly aligned, the bound recovers the matrix ARP factor \(r+1\). In the fully enumerated example, all assumptions are verified and this factor is attained up to floating-point accuracy. The mildly misaligned example shows that the bound remains valid but becomes less tight when the frequency-wise distributions differ.

The numerical experiments have two different goals. The synthetic, image, and video tests compare the practical quality of the selected subspaces. Both proposed methods perform favorably against the considered sampling baselines: ARP-T-CUR is particularly strong on the image experiments, while T-ARP is highly competitive on both real and synthetic data. These experiments use the two-sided projection $\mathcal C*\mathcal C^\dagger*\mathcal X*\mathcal R^\dagger*\mathcal R.$
A separate small synthetic experiment uses exactly the randomized interpolatory approximation studied in Theorem~\ref{thm:tarp-error}. By enumerating all possible index sequences, it verifies the assumptions and the corresponding expected-error bound.

Several questions remain open. First, it would be useful to identify classes of tensors for which the alignment constants remain close to one and the rank-growth condition holds. Second, the two-stage tensor cross bound could be tested with an experiment that uses exactly the approximation in Theorem~\ref{thm:tcross-error}. Other directions include adaptive rank selection, faster Householder or sketching implementations, and extensions to higher-order tensor products.

\bibliographystyle{elsarticle-num} 
\bibliography{ref}

 
\end{document}